\documentclass[a4paper,12pt]{amsart}
\usepackage{amsfonts}
\usepackage{mathrsfs}
\usepackage{amsmath,amssymb,latexsym,amsfonts,amscd, yfonts, amsthm, pifont}
\usepackage{cases}
\usepackage{bm}

\usepackage{enumitem}
\usepackage[all]{xy}
\allowdisplaybreaks[4]

\newtheorem{thm}{Theorem}[section]
\newtheorem{lem}[thm]{Lemma}
\newtheorem{cor}[thm]{Corollary}

\newtheorem{prop}[thm]{Proposition}
\newtheorem{claim}[thm]{Claim}

\newtheorem{conj}[thm]{Conjecture}

\theoremstyle{definition}

\newtheorem{definition}[thm]{Definition}

\newtheorem{example}[thm]{Example}

\theoremstyle{remark}
\newtheorem{remark}[thm]{Remark}

\numberwithin{equation}{section}

\newcommand{\bC}{{\mathbb C}}
\newcommand{\bA}{{\mathbb A}}

\newcommand{\bQ}{{\mathbb Q}}
\newcommand{\bP}{{\mathbb P}}

\newcommand{\rounddown}[1]{\lfloor{#1}\rfloor}

\newcommand\ZZ{{\mathbb{Z}}}

\newcommand\Supp{{\text{\rm Supp}}}

\newcommand\mult{{\text{\rm mult}}}

\newcommand\bR{{\mathbb{R}}}

\newcommand\mld{{\text{\rm mld}}}
\newcommand\mm{{\mathfrak{m}}}
\newcommand\xx{{\bf {x}}}

\newcommand\bmu{{\bm \mu}}

\title{A gap theorem for minimal log discrepancies of non-canonical singularities in dimension three}
\date{June 7, 2019}

\author{Chen Jiang}
\address{Shanghai Center for Mathematical Sciences, Fudan University, Jiangwan Campus, 2005 Songhu Road, Shanghai, 200438, China}
\email{chenjiang@fudan.edu.cn}

\begin{document}
\numberwithin{equation}{section}
\begin{abstract} 
We show that there exists a positive real number $\delta>0$ such that for any normal quasi-projective $\bQ$-Gorenstein $3$-fold $X$, if $X$ has worse than canonical singularities, that is, the minimal log discrepancy of $X$ is less than $1$, then the minimal log discrepancy of $X$ is not greater than $1-\delta$. As applications, we show that the set of all non-canonical klt Calabi--Yau $3$-folds are bounded modulo flops, and the global indices of all klt Calabi--Yau $3$-folds are bounded from above.
\end{abstract}
\maketitle

\pagestyle{myheadings} \markboth{\hfill C. Jiang
\hfill}{\hfill A gap theorem for mlds of non-canonical singularities in dimension $3$\hfill}
\tableofcontents

\section{Introduction}
Throughout this paper, we work over the complex number field $\bC$.

Canonical and terminal singularities, introduced by Reid, appear naturally in the minimal model program and play important roles in the birational classification of higher dimensional algebraic varieties. Such singularities are well-understood in dimension $3$, while the property of non-canonical singularities is still mysterious. In this paper, we investigate the difference between canonical and non-canonical singularities via minimal log discrepancies.

The minimal log discrepancy (mld) of a normal quasi-projective $\bQ$-Gorenstein  variety $X$, introduced by Shokurov, is defined to be the infimum of log discrepancies of all prime divisors on all birational models of $X$. It is an important invariant for singularities in the minimal model program, and is known to be related to the termination of flips and other topics of interest, see \cite{letter5, BS}. 
Here we recall the following deep conjecture regarding the behavior of minimal log discrepancies proposed by Shokurov. 
\begin{conj}[ACC for minimal log discrepancies, cf. {\cite[Problem 5]{Sho88}, \cite[Conjecture 4.2]{MR1420223}}]\label{c.mld.strong}
Fix a positive integer $d$ and a DCC set $I \subset [0,1]$. Then the set
\[
\{\mathrm{mld}_{\eta_Z}(X, \Delta) \mid (X, \Delta) \text{ is lc}, \; \dim X\leq d, \;
Z\subset X, \; \mathrm{coeff}(\Delta) \in I \}
\]
satisfies the ACC. 
\end{conj}

Here ACC stands for the ascending chain condition whilst DCC stands for the descending chain condition. 

 Conjecture \ref{c.mld.strong} is proved in dimension 2 by Alexeev \cite{Alexeev93} and Shokurov \cite{Sho91}, and
for toric pairs by Borisov \cite{Borisov} and Ambro \cite{Ambro}. Although some partial results are known \cite{Kaw14, Nak16, MN18, Kaw18, Liu18, Mor18, HLS19}, Conjecture \ref{c.mld.strong} still remains open in its full generality in dimensions $3$ and higher.

Recall that for a normal quasi-projective $\bQ$-Gorenstein  variety $X$, $\mld(X)\geq 1$ if and only if $X$ has canonical singularities. Hence in this paper, we are only interested in the following special case of Conjecture \ref{c.mld.strong}.

\begin{conj}[1-gap conjecture for minimal log discrepancies]\label{c.mld}Fix a positive integer $d$.
Then $1$ is not an accumulation point from below for the set of minimal log discrepancies of all normal quasi-projective $\bQ$-Gorenstein varieties of dimension $d$.\end{conj}

Conjecture \ref{c.mld} asserts that there is a gap for minimal log discrepancies between canonical and non-canonical singularities, and it already has interesting applications related to the boundedness of Calabi--Yau varieties (see \cite{rccy3}). Note that in Conjecture \ref{c.mld}, we are interested in the global minimal log discrepancies rather than the local ones at closed points. Although it is much weaker than Conjecture \ref{c.mld.strong}, Conjecture \ref{c.mld} was still open even in dimension $3$.

As the main result of this paper, we give an affirmative answer to Conjecture \ref{c.mld} in dimension $3$.

\begin{thm}
\label{gap main}
There exists a positive real number $\delta>0$ with the following property: if $X$ is a normal quasi-projective $\bQ$-Gorenstein $3$-fold with $\mld(X)<1$, then $\mld(X)\leq 1-\delta$.
\end{thm}

\begin{remark}
We explain the strategy of proving Theorem \ref{gap main} briefly. The goal is to show that there is no $3$-fold $X$ with $1-\delta<\mld(X)<1$ for a sufficiently small $\delta>0$. The first step is to reduce to the case that all but one exceptional divisors over $X$ have log discrepancies greater than $1$, in which case $X$ is called {\it extremely non-canonical}  (see Section \ref{reduction ENC}). Also it is easy to reduce to the case that $X$ is an isolated singularity which is a hyperquotient of an isolated cDV singularity in $\bA^4$. To deal with this case, we replay the game for the classification of $3$-dimensional terminal singularities by Mori \cite{Mori85} as explained by Reid \cite{YPG}, and show that such a singularity does not exist. Of course in our situation rules are changed which makes the game more complicated, but it will be in control after some essential modifications (see Section \ref{section HQ} for more explanations).
\end{remark}

\begin{remark}In many applications, it suffices to know the existence of such a positive number $\delta$. But by our method, it is possible to determine the number $\delta$ in Theorem \ref{gap main} effectively.
In fact, by the proof of Theorem  \ref{gap main}, we can take $\delta=\delta_0$, where $\delta_0$ is a positive constant given in Lemma \ref{non-canonical lemma} which is related to the gap of minimal log discrepancies of isolated cyclic quotient singularities in dimensions $3$
and $5$.
After the first version of this paper appeared on arXiv, the author was informed by Liu and Xiao \cite{LiuXiao} that they computed that $\delta_0=\frac{1}{19}$ in Lemma \ref{non-canonical lemma}, which then gives an optimal value $\delta=\frac{1}{13}$ for Theorem \ref{gap main} after some extra effort. 
\end{remark}

Next we explain the applications of Theorem \ref{gap main} to boundedness problem for singular Calabi--Yau $3$-folds.

A normal projective variety $X$ is a {\it Calabi--Yau} variety if $K_X\equiv 0$. According to the minimal model program, Calabi--Yau varieties form a fundamental class in birational geometry as building blocks of algebraic varieties. Calabi--Yau varieties are also interesting objects in differential geometry and mathematical physics.
Hence, it is interesting to ask whether such kind of varieties satisfies any finiteness properties, namely, whether some invariants of them are in a finite set, or they can be parametrized by finitely many families. For recent developments on this direction in birational geometry, see \cite{Ale94, AM, DCS, rccy3, Birkar18}. We recall that Alexeev \cite[Theorem 6.9]{Ale94} showed that all 
 {Calabi--Yau} varieties in dimension $2$ with worse than du Val singularities form a bounded family. Motivated by Alexeev's result, \cite{rccy3} considers rationally connected klt Calabi--Yau $3$-folds and showed their boundedness modulo flops assuming Theorem \ref{gap main}.

As an application of Theorem \ref{gap main}, we show that the set of all non-canonical klt Calabi--Yau $3$-folds are bounded modulo flops, which is  a weak version of the analogue of Alexeev's result in dimension $3$.

\begin{thm}[=Theorem \ref{bdd flop cy}]\label{main 3}
The set of non-canonical klt Calabi--Yau $3$-folds forms a bounded family modulo flops.
\end{thm}

Note that \cite{rccy3} only considers Theorem \ref{main 3} for rationally connected klt Calabi--Yau $3$-folds, but we are able to remove the rational connectedness condition in this paper.

As a consequence, the global indices of all klt Calabi--Yau $3$-folds are bounded from above.
\begin{cor}[=Corollary \ref{bdd index cy}]\label{main 2}
There exists a positive integer $m$ such that for any klt Calabi--Yau $3$-fold $X$, $mK_X\sim 0$.
\end{cor}

Here we remark that Theorem \ref{main 2} was known for canonical Calabi--Yau $3$-folds by Kawamata \cite{K=0} and Morrison \cite{Morrison}. So we only need to deal with the case of non-canonical klt Calabi--Yau $3$-folds, which follows from Theorem \ref{main 3}.
Also we recall that Blache and Zhang \cite{Blache, DQ1, DQ2} studied klt Calabi--Yau surfaces 
(also known as log Enriques surfaces) and showed that for any such surface $S$, $mK_S\sim 0$ for some $m\leq 21$. So Corollary \ref{main 2} is a generalization of  their results in dimension $3$. 
Of course it is very interesting to ask for an effective bound of the global indices, but our method can not give an effective bound.

It is worthwhile to mention that Jingjun Han brought our attention to another application of Theorem \ref{gap main}, which is the termination of log twists (introduced by Birkar and Shokurov) in dimension $3$. See \cite[Proposition 3.4]{BS} for details.


\bigskip

This paper is organized as follows. In Section \ref{prelim}, we recall basic definitions and make preparation for the proof.  In Section \ref{reduction ENC}, we reduce Conjecture \ref{c.mld} to the case of extremely non-canonical singularities.
In Section \ref{section HQ}, we prove Theorem \ref{gap main} for $3$-dimensional isolated hyperquotient extremely non-canonical singularities, using the method from classification of $3$-dimensional terminal singularities. In Section \ref{section general case}, we  prove Theorem \ref{gap main} for the general case. In Section \ref{section bounded}, we prove Theorem \ref{main 3} and Corollary \ref{main 2}.

\section{Preliminaries}\label{prelim}
We adopt the standard notation and definitions in \cite{KMM} 
and \cite{KM}, and will freely use them. We work over $\bC$.

\subsection{Residues of integers}
For a positive integer $r$, $(\overline{n})_r$ denotes the smallest non-negative residue modulo $r$, i.e., the number $m$ such that $0\leq m<r$ and $n\equiv m \bmod r$. Usually $r$ is clear in the context, so we simply write $\overline{n}$ instead of $(\overline{n})_r$.
We will often use the following easy fact:
 $\overline{n}+\overline{-n}=\begin{cases}r & \text{ if } n\not\equiv 0 \bmod r;\\
 0 & \text{ if } n\equiv 0 \bmod r.
 \end{cases}$
 

\subsection{Pairs, singularities, and minimal log discrepancies}
A {\it log pair} $(X, B)$ consists of a normal quasi-projective variety $X$ 
and an effective $\bR$-divisor $B$ on $X$ such that $K_X+B$ is $\bR$-Cartier.

Let $f\colon Y\rightarrow X$ be a log
resolution of the log pair $(X, B)$. Write
\[
K_Y =f^*(K_X+B)+\sum a_iF_i,
\]
where $\{F_i\}$ are distinct prime divisors. 
For a non-negative real number $\epsilon$, the log pair $(X,B)$ is called
\begin{itemize}
\item[(a)] \emph{kawamata log terminal} (\emph{klt}
for short) if $a_i> -1$ for all $i$;
\item[(b)] \emph{$\epsilon$-log canonical} (\emph{$\epsilon$-lc} for
short) if $a_i\geq -1+\epsilon$ for all $i$;
\item[(c)] \emph{terminal} if $a_i> 0$ for all $f$-exceptional divisors $F_i$ and all $f$;
\item[(d)] \emph{canonical} if $a_i\geq 0$ for all $f$-exceptional divisors $F_i$ and all $f$;
\item[(e)] \emph{purely log terminal} (\emph{plt}
for short) if $a_i\geq 0$ for all $f$-exceptional divisors $F_i$ and all $f$.
\end{itemize}

Usually we write $X$ instead of $(X,0)$ in the case when $B=0$. In this case, when we talk  about singularities as above, we automatically assume that $X$ is $\bQ$-Gorenstein, that is, $K_X$ is $\bQ$-Cartier.
Note that we usually use lc instead of $0$-lc.
Also note that 
$\epsilon$-lc singularities only make sense if $\epsilon\in [0,1]$. 
 
The {\it log discrepancy} of the divisor $F_i$ is defined to be 
$$a(F_i; X, B)=\mult_{F_i}(K_Y-f^*(K_X+B))+1= a_i+1.$$
It does not depend on the choice of the log resolution $f$. Here we identify divisors on different birational models by its divisorial valuation.
When $B=0$, we simply write $a(F_i; X)$ instead of $a(F_i; X, B)$.


Let $(X, B)$ be a log pair and $Z \subset X$ an irreducible closed subset with $\eta_Z$ the generic point of $Z$. 
The {\it minimal log discrepancy} of $(X, B)$ over $Z$ is defined as
\[
\mathrm{mld}_Z(X, B)=\inf_E\{a(E;X,B)\mid \text{center}_{X}(E)\subset Z\},\]
and the {\it minimal log discrepancy} of $(X, B)$ at $\eta_Z$ is defined as
\[
\mathrm{mld}_{\eta_Z}(X, B)=\inf_E\{a(E;X,B)\mid \text{center}_{X}(E)= Z\}.\]
Moreover, we write $\mathrm{ld}(X, B)$ instead of $\mathrm{mld}_X(X, B)$, and call it the {\it total log discrepancy} of $(X, B)$.
We define the  the {\it minimal log discrepancy} of $(X, B)$ to be $\mathrm{mld}(X, B)=\inf_Z\mathrm{mld}_{Z}(X, B)$ where $Z$ runs over all subvarieties of codimension $2$. Note that the difference between  total log discrepancy and minimal log discrepancy is just whether codimension $1$ points (or prime divisors) on $X$ are considered or not.  If $B=0$, we simply write $\mathrm{ld}(X)$ and $\mathrm{mld}(X)$.
Note that if $\mathrm{mld}(X)\leq 1$, then $\mathrm{ld}(X)=\mathrm{mld}(X)$.

Note that $\mathrm{ld}(X, B)\geq \epsilon$ (resp. $>0$) if and only if $(X, B)$ is $\epsilon$-lc (resp. klt), and 
$\mathrm{mld}(X, B)\geq 1$ (resp. $>0$)  if and only if $(X, B)$ is canonical  (resp. terminal). So $X$ is {\it non-canonical} if and only if
$\mathrm{mld}(X)< 1$.

\subsection{Log Calabi--Yau pairs}

A normal projective variety $X$ is a {\it Calabi--Yau} variety if $K_X\equiv 0$. If $K_X\sim_\bQ 0$, then the {\it global index} of $X$ is the minimal positive integer $m$ such that $mK_X\sim 0$.

A log pair $(X, B)$ is called a
\emph{log Calabi--Yau pair} if $X$ is projective and $K_X+B\equiv 0$. 
Recall that if $(X, B)$ is lc, this is equivalent 
to $K_X+B\sim_\bR 0$ by \cite{G}.

\subsection{Bounded pairs}\label{sec.bdd}
A collection of projective varieties $ \mathcal{D}$ is
said to be \emph{bounded} (resp., \emph{bounded in codimension one})
if there exists  a projective morphism 
$h\colon \mathcal{Z}\rightarrow S$
between schemes of finite type such that
each $X\in \mathcal{D}$ is isomorphic (resp., isomorphic in codimension one) to $\mathcal{Z}_s$ 
for some closed point $s\in S$  where $\mathcal{Z}_s=h^{-1}(s)$.

We say that a collection of projective log pairs $\mathcal{D}$ is \emph{log bounded} (resp., \emph{log bounded in codimension one})
if there is a quasi-projective scheme $\mathcal{Z}$, a 
reduced divisor $\mathcal{E}$ on $\mathcal Z$, and a 
projective morphism $h\colon \mathcal{Z}\to S$, where 
$S$ is of finite type and $\mathcal{E}$ does not contain 
any fiber, such that for every $(X,B)\in \mathcal{D}$, 
there is a closed point $s \in S$ and a birational
map $f \colon \mathcal{Z}_s \dashrightarrow X$ which is isomorphic
(resp., isomorphic in codimension one)
such that $\mathcal{E}_s:=\mathcal{E}|_{\mathcal{Z}_s}$ 
coincides with the support of $f_*^{-1}B$.

Moreover, if $\mathcal{D}$ is a set of klt Calabi--Yau 
varieties (resp., klt log Calabi--Yau pairs), then it is 
said to be {\it bounded modulo flops} (resp., {\it log 
bounded modulo flops}) if it is (log) bounded in 
codimension one, each fiber $\mathcal{Z}_{s}$ 
corresponding to $X$ in the definition is normal projective, 
and $K_{\mathcal{Z}_s}$ is $\bQ$-Cartier (resp., $K_{\mathcal{Z}_s}+f_*^{-1}B$ is $\bR$-Cartier). 

Note that if $\mathcal{D}$ is a set of klt log Calabi--Yau 
pairs which is log bounded modulo flops, and 
$(X, B)\in \mathcal{D}$ with a birational
map $f \colon \mathcal{Z}_s \dashrightarrow X$ 
isomorphic in codimension one as in the definition, 
then $(\mathcal{Z}_s, f_*^{-1}B)$ is also a klt log 
Calabi--Yau pair by the negativity lemma. 
Moreover, $(X, B)$ is $\epsilon$-lc if and only 
if $(\mathcal{Z}_s, f_*^{-1}B)$ is $\epsilon$-lc. A similar statement 
holds when $\mathcal{D}$ is a set of klt Calabi--Yau varieties.

Here the name ``modulo flops" comes from the fact that, if we assume that $X$ and $\mathcal{Z}_s$ are both $\bQ$-factorial, then they are connected by flops by running a $(K_X+B+\delta f_*H)$-MMP where $H$ is an ample divisor on $\mathcal{Z}_s$ and $\delta$ is a sufficiently small positive number (\cite{BCHM, flops}).

\subsection{Extremely non-canonical singularities}
As we are interested in non-canonical singularities, we introduce the concept of extremely non-canonical singularities, which are the closest to  terminal singularities among all non-canonical singularities.
\begin{definition}
Let $X$ be a normal quasi-projective variety.
We say that $X$ is {\em extremely non-canonical} if $X$ has $\bQ$-factorial klt singularities and
\begin{enumerate}
\item there exists exactly one prime divisor $E_0$ over $X$ such that $a(E_0; X)<1$;
\item there is no divisor $E$ over $X$ with $a(E; X)=1$.
\end{enumerate}
\end{definition}

\begin{remark}
Suppose that $X$ is extremely non-canonical, then it is easy to see that $a(E_0; X)=\mld(X)$, and $X$ has terminal singularities outside the center of $E_0$ on $X$. 
\end{remark}

\subsection{Cyclic quotient singularities and hyperquotient singularities}\label{subsec cyclic quotient} 
We recall the concept of hyperquotient singularities and the toric method which are useful in the classification of $3$-dimensional terminal singularities. Most of the contents come from \cite[Section 4]{YPG} except for Theorem \ref{key thm}.

Let $r$ be a positive integer. Let $\bmu_r$ denote the cyclic group of $r$-th roots of unity in $\bC$. A {\em cyclic quotient singularity} is of the form $\mathbb{A}^{n+1}/\bmu_r$, where the action of $\bmu_r$ is given by 
\[
\bmu_r\ni \xi: (x_0, \ldots, x_n)\mapsto (\xi^{a_0}x_0, \ldots,\xi^{a_n} x_n)
\]
for certain $a_0, \ldots, a_n\in \mathbb{Z}/r$. Note that we may always assume that the action of $\bmu_r$ on $\mathbb{A}^{n+1}$ is small, that is, it contains no reflection (\cite[Definition 7.4.6, Theorem 7.4.8]{Ishii}). We say that $\mathbb{A}^{n+1}/\bmu_r$ is of {\em type $\frac{1}{r}(a_0,\ldots, a_n)$}. Recall that this singularity is isolated if and only if $\gcd(a_i, r)=1$ for every $0\leq i\leq n$ by \cite[Remark 1]{Fujiki}.

The toric geometry interpretation of cyclic quotient singularities is as following (\cite[(4.3)]{YPG}):
let $\overline{M}\simeq \ZZ^{n+1}$ be the lattice of monomials on $\bA^{n+1}$, and $\overline{N}$ its dual. Define $N$ by $N=\overline{N}+\ZZ\cdot \frac{1}{r}(a_0, \ldots, a_n)$ and $M\subset \overline{M}$ the dual sublattice. Let $\sigma=\bR_{\geq 0}^{n+1}\subset N_\bR$ be the positive quadrant and $\sigma^\vee \subset M_\bR$ the dual quadrant. Then in toric geometry, 
\[
\bA^{n+1}=\text{Spec}~\bC[\overline{M}\cap \sigma^\vee]
\]
and its quotient
\[
\mathbb{A}^{n+1}/\bmu_r=\text{Spec}~\bC[{M}\cap \sigma^\vee]=T_N(\Delta),
\]
where $\Delta$ is the fan corresponding to $\sigma$.

Now we are interested in the hypersurface singularity $(Q\in Y): (f=0) \subset \bA^{n+1}$ with an action of $\bmu_r$ which is free outside $Q$ and its quotient $(P\in X)=Y/\bmu_r$. It is known that the action of $\bmu_r$ extends to a small $\bmu_r$-action of $\bA^{n+1}$ (\cite[Lemma 8.3.8]{Ishii}). We still assume that the action of $\bmu_r$ on $\bA^{n+1}$ is given by 
\[
\bmu_r\ni \xi: (x_0, \ldots, x_n)\mapsto (\xi^{a_0}x_0, \ldots,\xi^{a_n} x_n).
\]
As $Y=(f=0)$ is fixed by the action of $\bmu_r$, we may write 
\[
\bmu_r\ni \xi: f\mapsto \xi^{e}f
\]
for certain $e\in \mathbb{Z}/r$. Such $(P\in X)$ is called a {\em hyperquotient singularity of type $\frac{1}{r}(a_0, \ldots, a_n; e)$}. Note that the action of $\bmu_r$ on the generator 
\[
s=\text{Res}_Y\frac{dx_0\wedge \cdots \wedge dx_n}{f}=\frac{dx_1\wedge \cdots \wedge dx_n}{\partial f/\partial x_0}\in \omega_Y
\]
is given by 
\[
\bmu_r\ni \xi: s \mapsto \xi^{a_0+\cdots +a_n-e}s.
\]

Let $\alpha=(b_0, \ldots, b_n)\in N\cap \sigma$ be a vector, that is, $\alpha$ is a weighting such that $\alpha(x_i)=b_i\in \bQ$ on monomials such that 
\begin{enumerate}
\item $\alpha\in N$, that is, $\alpha\equiv \frac{1}{r}(ja_0, \ldots, ja_n) \bmod \ZZ^{n+1}$ for some $j=0,1,\ldots, r-1$;
\item $\alpha\in \sigma$, that is, $b_i\geq 0$ for all $i$.
\end{enumerate}
This weighting can be extended to $\bC[x_0, \ldots, x_n]$ in the following way:
for $\xx^{\bf m}=x_0^{m_0}\cdots x_n^{m_n}$, $\alpha(\xx^{\bf m})=\sum_{i=0}^n m_i\alpha(x_i)=\sum_{i=0}^n m_ib_i$; and for a polynomial $f\in \bC[x_0, \ldots, x_n]$,
\[
\alpha(f):=\min\{\alpha(\xx^{\bf m})\mid \xx^{\bf m}\in f\}.
\] 
Here $\xx^{\bf m}\in f$ means that the monomial $\xx^{\bf m}$ appears in $f$ with non-zero coefficient.

\begin{prop}\label{blow up a} Consider $(Q\in Y): (f=0) \subset \bA^{n+1}$ with an action of $\bmu_r$ which is free outside $Q$ and its quotient $(P\in X)=Y/\bmu_r$. Keep the above notation. Let $\alpha\in N\cap \sigma$ be a primitive vector and $\Delta(\alpha)$ be the star-shaped subdivision of $\Delta$ by $\alpha$, then the toric morphism $\phi_{\alpha}: T_N(\Delta(\alpha))\to T_N(\Delta)=\mathbb{A}^{n+1}/\bmu_r$ extracts an exceptional divisor $E_{\alpha}$. Denote $Z_\alpha=T_N(\Delta(\alpha))$ and $Z=T_N(\Delta)$, and let $X_\alpha\subset Z_\alpha$ be the strict transform of $X$ on $Z_\alpha$. Then 
\begin{enumerate}
\item $K_{Z_\alpha}=\phi_{\alpha}^*K_Z+(\alpha(x_0\cdots x_n)-1)E_\alpha$;
\item $X_\alpha=\phi_{\alpha}^*X-\alpha(f)E_\alpha$.
\end{enumerate}
\end{prop}
\begin{proof}This is standard, see \cite[(4.8)]{YPG} or \cite[Proposition 8.3.11]{Ishii}.
\end{proof}

In the classification of $3$-dimensional terminal singularities, Proposition \ref{blow up a} is used to provide a necessary condition for a hyperquotient singularity being terminal (see \cite[(4.6) Theorem]{YPG}). As we are considering non-canonical singularities, in the following theorem   we provide  a necessary condition for an isolated hyperquotient singularity being extremely non-canonical  by the toric method, which plays an essential role in the proof in Section \ref{section HQ}.

\begin{thm}\label{key thm} Fix $0<\delta\leq \frac{1}{2}$. Consider $(Q\in Y): (f=0) \subset \bA^{n+1}$ with an action of $\bmu_r$ which is free outside $Q$ and its quotient $(P\in X)=Y/\bmu_r$. 
Assume further that $(P\in X)$ is an isolated extremely non-canonical singularity with $\mld(X)>1-\delta$. Keep the above notation. Then 
\begin{enumerate}
\item there exists at most one primitive vector $\beta\in N\cap \sigma$ such that 
\[
1-\delta< \beta(x_0\cdots x_n)-\beta(f)<1;
\]
\item for any primitive vector $\alpha\in N\cap \sigma$ such that $\alpha\neq \beta$,
\[
 \alpha(x_0\cdots x_n)-\alpha(f)>1.
\]
\end{enumerate}
Furthermore, for any vector $\alpha'\in N\cap \sigma$ such that $\alpha'\neq \beta$,
\[
 \alpha'(x_0\cdots x_n)-\alpha'(f)>1.
\]
\end{thm}
\begin{proof}
Assume that there exists a primitive vector $\beta\in N\cap \sigma$ such that $\beta(x_0\cdots x_n)-\beta(f)\leq 1$. To see the first two statements, it suffices to show that such $\beta$ is unique and 
\[
1-\delta<\beta(x_0\cdots x_n)-\beta(f)<1.
\] 

Keep the notation in Proposition \ref{blow up a}, we have 
\[
K_{Z_\beta}+X_{\beta}=\phi_\beta^*(K_Z+X)+(\beta(x_0\cdots x_n)-\beta(f)-1)E_\beta,
\] 
which can be rewritten as 
\[
K_{Z_\beta}+X_{\beta}+tE_\beta=\phi_\beta^*(K_Z+X),
\]
where $t=1+\beta(f)-\beta(x_0\cdots x_n)\geq 0$.
Since $X$ has an isolated klt singularity, the pair $(Z, X)$ is plt by inversion of adjunction, which implies that $({Z_\beta}, X_{\beta}+tE_\beta)$ is also plt. By the subadjunction formula (\cite[16.6 Proposition, 16.7 Corollary]{Kollar92}), there is a boundary $B_\beta$ on $X_{\beta}$ such that 
\[
K_{X_\beta}+B_\beta=(K_{Z_\beta}+X_{\beta}+tE_\beta)|_{X_{\beta}}=\phi_\beta|_{X_\beta}^*(K_X),
\]
and the coefficients of $B_\beta$ are of the form $1-\frac{1}{l}+\frac{kt}{l}$ for some positive integers $l, k$, here $k>0$ since $E_\beta$ intersects $X_\beta$. By the assumption that $X$ is extremely non-canonical, coefficients of $B_\beta$ are positive since there is no exceptional divisor over $X$ with log discrepancy $1$, and in fact $B_\beta$ has exactly one component $F_\beta$ with coefficient $1-\frac{1}{l}+\frac{kt}{l}>0$.
Since $\mld(X)>1-\delta$,  $1-\frac{1}{l}+\frac{kt}{l}<\delta\leq \frac{1}{2}$. In particular, $l=1$ and $0<t< \delta$. This shows that 
\[
1-\delta<\beta(x_0\cdots x_n)-\beta(f)<1.
\] 

To see the uniqueness of $\beta$, we look at the divisorial valuation $v_{F_\beta}$ on $\bC(X)$, and the following proof is suggested by Jungkai Chen. Since $l=1$, from the subadjunction formula, we get $E_\beta|_{X_{\beta}}=B_\beta=kF_\beta$. Hence $v_{F_\beta}(\xx^{\bf m})=k\beta(\xx^{\bf m})$ for any monomial $\xx^{\bf m}$ (${\bf m}\in M$). By the assumption that $X$ is extremely non-canonical, $v_{F_\beta}$ is unique. Hence such $\beta$ is unique by the primitivity.

For the last statement, for any non-primitive vector $\alpha'\in N\cap \sigma$, we may write  $\alpha'=m\alpha$ where $m\geq 2$ is an integer and $\alpha\in N\cap \sigma$ is primitive. Then \[
\alpha'(x_0\cdots x_n)-\alpha'(f)=m( \alpha(x_0\cdots x_n)-\alpha(f))> 2(1-\delta)\geq 1.
\]
\end{proof}

\subsection{ACC for minimal log discrepancies of cyclic quotient singularities}\label{subsection delta35}

Recall that the ACC for minimal log discrepancies is proved for toric varieties \cite{Borisov, Ambro}, in this paper we only need the following special case for cyclic quotient singularities:
\begin{thm}[{\cite{Borisov}}] Conjecture \ref{c.mld.strong} holds for cyclic quotient singularities. In particular, fix a positive integer $d$,
then the set of minimal log discrepancies of $d$-dimensional cyclic quotient singularities $(0\in W)$ at $0$ satisfies the ascending chain condition.
\end{thm}

As  corollaries, $2$ and $1$ are not accumulation points of these sets from below, and we will only use this fact in dimensions $3$ and $5$.
\begin{cor}\label{acc quot 3}
 There exists a positive constant $\delta_3>0$ such that for any isolated cyclic quotient singularity $(0\in W)$ in dimension $3$, if $\mld_0(W)<1$, then $\mld_0(W)\leq 1-\delta_3$.
\end{cor}

\begin{cor}\label{acc quot 5}
 There exists a positive constant $\delta_5>0$ such that for any cyclic quotient singularity $(0\in W)$ in dimension $5$, if $\mld_0(W)<2$, then $\mld_0(W)\leq 2-\delta_5$.
\end{cor}

Note that in Corollary \ref{acc quot 3} we are only interested in isolated singularities, but in Corollary 
 \ref{acc quot 5} the singularities are not necessarily isolated.

The following example is suggested by Alexeev:
\begin{example}\label{12/13}
Consider $(0\in W)$ to be a $3$-dimensional isolated cyclic quotient singularity of type $\frac{1}{13}(3,4,5)$, then $\mld(W)=\mld_0(W)=\frac{12}{13}$. In particular, $\delta_3\leq \frac{1}{13}$.
\end{example}
Here for the computation of minimal log discrepancies of toric varieties, we refer to \cite{Ambro} (see also the proof of Lemma \ref{non-canonical lemma}).

\begin{remark}
In fact, it is not difficult to show that Example \ref{12/13} is optimal, that is, we can take $\delta_3=\frac{1}{13}$ in Corollary \ref{acc quot 3}. This can be done after some tedious but elementary calculation by hand. We will not give the proof nor use this fact in this paper.
The value of $\delta_5$, on the other hand, seems to be more subtle as the dimension is higher and the singularities are not necessarily isolated. \end{remark}

\subsection{The terminal lemma and the non-canonical lemma}
In this subsection, we recall the terminal lemma by Morrison and Stevens \cite{MS} which plays an important role in the classification of $3$-dimensional terminal singularities. Here we only recall a special version for our application, for the full version we refer to \cite[(5.4) Theorem]{YPG}. Recall that $\overline{n}$ denotes the smallest non-negative residue modulo $r$.
\begin{thm}[{\cite[(5.4) Theorem, (5.6) Corollary]{YPG}}]\label{terminal lemma} Let $\frac{1}{r}(a_1, \cdots, a_4; e, 1)$ be a $6$-tuple of rational numbers with denominator $r$ such that 
$
q=\gcd(e, r)=\gcd(a_4, r)
$, and $a_1, a_2, a_3$ are coprime to $r$.
Assume that for $k=1, \ldots, r-1$,
\[
\sum_{i=1}^4\overline{a_ik}=\overline{ek}+k+r.
\]
If $q>1$, then $a_4\equiv e \bmod r$, and the remaining 4 elements can be paired together as $a_1\equiv 1, a_2+a_3\equiv 0 \bmod r$ (or permutations); if $q=1$, then $\{a_1, a_2, a_3, a_4,-e, -1\}$ can be split up into $3$ disjoint pairs which add to $0 \bmod r$ (for example, $a_1+a_2\equiv a_3+a_4\equiv -e-1\equiv 0 \bmod r$).
\end{thm}
\begin{remark}
Note that in the statement of \cite[(5.6) Corollary]{YPG}, the $q=1$ case is missing, but it can be easily derived from \cite[(5.4) Theorem]{YPG}.
\end{remark}
In order to study extremely non-canonical singularities by the toric method, we change the condition of the above terminal lemma and introduce the following ``non-canonical" lemma.
 
\begin{lem}\label{non-canonical lemma}There exists a positive real number $\delta_0\leq \delta_3<1$ with the following property.
 Let $\frac{1}{r}(a_1, \cdots, a_4; e)$ be a $5$-tuple of rational numbers with denominator $r$ such that 
$
q=\gcd(e, r)=\gcd(a_4, r)
$, and $a_1, a_2, a_3$ are coprime to $r$. Assume
one of the following holds:
\begin{enumerate} 
 \item[(\ding{73}1)] $a_1+a_2\equiv e\bmod r$;
 
  \item[(\ding{73}2)] $2a_4\equiv e\bmod r$;
 
 \item[(\ding{73}3)] $2a_1\equiv e \bmod r$ and $q\leq 2$.
 \end{enumerate}
Moreover, assume that
\begin{enumerate} 
 \item there exists a positive integer $k_0$ such that $1\leq k_0\leq r-1$ and
\[
\sum_{i=1}^4\overline{a_ik_0}= \overline{ek_0}+k_0;
\]

 \item for every integer $k$ such that $1\leq k\leq r-1$ and $k\neq k_0$,
\[
\sum_{i=1}^4\overline{a_ik}\geq \overline{ek}+ r.
\]
\end{enumerate}
Then $\frac{k_0}{r}\leq 1-\delta_0$. Here $\delta_3$ is the constant from Corollary \ref{acc quot 3}.
\end{lem}

\begin{proof}
We will show that we can take $\delta_0=\min\{\delta_3, \delta_5\}>0$. Here $\delta_3$ and $\delta_5$ are constants from Corollaries \ref{acc quot 3} and \ref{acc quot 5}.

Since $a_1, a_2, a_3$ are coprime to $r$, we know that $\overline{a_1k_0}, \overline{a_2k_0}, \overline{a_3k_0}$ are not $0$. Since $
\gcd(e, r)=\gcd(a_4, r)
$, $\overline{a_4k_0}=0$ if and only if $\overline{ek_0}=0$.

First assume that $\overline{a_4k_0}\neq 0$ and $\overline{ek_0}\neq 0$.
Consider $Z=\bA^5/\bmu_r$ to be a cyclic quotient singularity of type $\frac{1}{r}(a_1, \cdots, a_4, -e)$.
It suffices to show that $\mld_0(Z)=1+\frac{k_0}{r}<2$. Keep the notation in Subsection \ref{subsec cyclic quotient}. 
By the existence of log resolutions in toric category, we can compute the minimal log discrepancy by torus invariant divisors over $Z$. Recall that for the exceptional divisor $E_\alpha$ corresponding to a primitive vector $\alpha\in N\cap \sigma$, its log discrepancy is computed by $a(E_\alpha; Z)=\alpha(x_0\cdots x_4)$ (Proposition \ref{blow up a}). This means that (\cite{Ambro})
\[
\mld_0(Z)=\min\{\alpha(x_0\cdots x_4) \mid \alpha\in N\cap \text{relin}(\sigma)\},
\]
where $\text{relin}(\sigma)$ is the relative interior of $\sigma$.
By the assumption, we can consider 
\[
\beta=\frac{1}{r}(\overline{a_1k_0}, \ldots,\overline{a_4k_0},\overline{-ek_0} )\in N\cap \text{relin}(\sigma).
\]
Assumption (1) gives 
\[
\beta(x_0\cdots x_4)=\frac{1}{r} (\sum_{i=1}^4\overline{a_ik_0}+r- \overline{ek_0}) =1+\frac{k_0}{r}<2.
\]
On the other hand, take any $\alpha\in N\cap \text{relin}(\sigma)$ such that $\alpha\neq \beta$, 
recall that we can write $\alpha=(b_0, \ldots, b_4)$ such that $\alpha\equiv \frac{1}{r}(a_1j, \ldots, a_4j, -ej) \bmod \ZZ^{5}$ for some $j=0,1,\ldots, r-1$ and $b_i>0$ for all $i$. If $j=0$, then $\alpha(x_0\cdots x_4)\geq 5$. If $j=k_0$, then $\alpha(x_0\cdots x_4)\geq \beta(x_0\cdots x_4)+1\geq 2$. If $1\leq j\leq r-1$ and $j\neq k_0$, then since  $b_4>0$, we know that $b_4\geq \frac{1}{r}(r-\overline{je})$, and by  assumption (2), 
\[
\alpha(x_0\cdots x_4)\geq \frac{1}{r} (\sum_{i=1}^4\overline{a_ij}+r- \overline{ej}) \geq 2.
\]
Hence $\mld_0(Z)=1+\frac{k_0}{r}<2$. By Corollary \ref{acc quot 5}, $\frac{k_0}{r}\leq 1-\delta_5$.

Then assume that $\overline{a_4k_0}=\overline{ek_0}= 0$. Denote $q=\gcd(e, r)=\gcd(a_4, r)$ and $r=pq$. Then $p$ divides $k_0$ and we can write $k_0=pk'_0$. Now let $(\overline{n})_q$ be the smallest non-negative residue of $n$ modulo $q$, then $p(\overline{n})_q=\overline{pn}$. Hence we get new relations for $\frac{1}{q}(a_1,a_2,a_3)$ and $k'_0$ as the following:
for every integer $k'$ such that $1\leq k'\leq q-1$ and $k'\neq k'_0$,
\[
\sum_{i=1}^3(\overline{a_ik'})_q=\frac{1}{p}\sum_{i=1}^3\overline{a_ipk'}=\frac{1}{p}(\sum_{i=1}^4\overline{a_ipk'}-\overline{epk'})\geq q;
\]
on the other hand,
\[
\sum_{i=1}^3(\overline{a_ik'_0})_q=\frac{1}{p}\sum_{i=1}^3\overline{a_ipk'_0}=\frac{1}{p}(\sum_{i=1}^4\overline{a_ipk'_0}-\overline{epk'_0})= k'_0.
\]
Now we can consider $Z'=\bA^3/\bmu_q$ to be a cyclic quotient singularity of type $\frac{1}{q}(a_1, a_2, a_3)$. It is isolated since $a_1, a_2, a_3$ are coprime to $r$.
By the same calculation as above, $\mld_0(Z')=\frac{k'_0}{q}=\frac{k_0}{r}<1$. To be more precise, 
we can consider 
\[
\beta'=\frac{1}{q}((\overline{a_1k'_0})_q, (\overline{a_2k'_0})_q, (\overline{a_3k'_0})_q )\in N\cap \text{relin}(\sigma),
\]
here $N$ is the lattice corresponding to $Z'$ by abusing the notation.
Then 
\[
\beta'(x_0x_1x_2)=\frac{1}{q}  \sum_{i=1}^3(\overline{a_ik'_0})_q =\frac{k'_0}{q}<1.
\]
On the other hand, take any $\alpha\in N\cap \text{relin}(\sigma)$ such that $\alpha\neq \beta'$, 
recall that we can write $\alpha=(b_0, b_1, b_2)$ such that $\alpha\equiv \frac{1}{r}(a_1j, a_2j, a_3j) \bmod \ZZ^{3}$ for some $j=0,1,\ldots, q-1$ and $b_i>0$ for all $i$. We may assume that $b_i<1$ for all $i$, otherwise $\alpha(x_0x_1x_2)\geq 1$. Hence 
$\alpha= \frac{1}{q}((\overline{a_1j})_q, (\overline{a_2j})_q, (\overline{a_3j})_q)$ with $1\leq j\leq q-1$ and $j\neq k'_0$. In this case,  
\[
\alpha(x_0x_1x_2)= \frac{1}{q} \sum_{i=1}^3(\overline{a_ij})_q \geq 1.
\]
Hence $\mld_0(Z')=\frac{k'_0}{q}<1$. 
By Corollary \ref{acc quot 3}, $\frac{k_0}{r}\leq 1-\delta_3$.
\end{proof}

\begin{remark}\label{remark LX}
In the proof of Lemma \ref{non-canonical lemma}, assumptions (\ding{73}1--3) are not used. But we keep these assumptions for two reasons. For one thing, we always get 
one of (\ding{73}1--3) in our applications (see Propositions  \ref{prop xy}  and \ref{prop x2}). For the other, these assumptions will be helpful when one tries to find an optimal or effective value for $\delta_0$ (and $\delta$ in Theorem \ref{gap main}). 
In fact, in a recent preprint by Liu and Xiao \cite{LiuXiao}, they show that $\delta_0=\frac{1}{19}$ in Lemma \ref{non-canonical lemma} by some clever arguments with a help of computer program.
\end{remark}

\section{Reduction to extremely non-canonical singularities}\label{reduction ENC}
In this section, we   reduce the 1-gap conjecture to the case of extremely non-canonical singularities. During the preparation of this paper, we are informed by Jingjun Han and Jihao Liu that they also got similar result as Theorem \ref{X<Y} independently. 

\begin{thm}\label{X<Y}
Let $X$ be a normal quasi-projective variety with klt singularities such that $\mld(X)<1$. Then there exists a projective birational morphism $Y\to X$ such that $Y$ is extremely non-canonical and $\mld(X)\leq \mld(Y)<1$.
\end{thm}

\begin{proof}
Let $X$ be a normal quasi-projective variety with klt singularities such that $\mld(X)<1$. 
Take $\mathfrak{E}$ to be the set of all exceptional prime divisors $E$ over $X$ with $a(E; X)\leq 1$, which is a finite set by \cite[Proposition 2.36]{KM}.
By \cite[Corollary 1.4.3]{BCHM}, there exists a projective birational morphism $\pi: W\to X$ with $W$ $\bQ$-factorial such that   $\mathfrak{E}$ is the set of exceptional divisors of $\pi$.
We may write 
$$K_{W}+\Delta=\pi^*K_X$$
where $\Delta$ is a non-zero effective $\bQ$-divisor as $\mld(X)<1$.  Note that $(W, \Delta)$ is canonical by the construction.
For a sufficiently small $\epsilon>0$, by \cite{BCHM}, we can run a $(W, (1+\epsilon)\Delta)$-MMP over $X$, which terminates and reaches a minimal model over $X$ contracting $\text{Supp}(\Delta)$. Denote $W'\to X$ to be the model obtained by the first divisorial contraction in this MMP. We will show that $W'$ satisfies the requirement of the theorem.

Denote $E_1$ to be the prime divisor on $W$ contracted on $W'$, and denote $\Delta'$ to be the strict transform of $\Delta$ on $W'$. Take a common resolution $p: Z\to W$, $p': Z\to W'$. As this MMP is also a $\Delta$-MMP over $X$,
we can write
$$
p^*\Delta=p'^*\Delta'+F
$$
where $F$ is an effective $\bQ$-divisor and $\mult_{E_1}(p_*F)>0$.
On the other hand, we have 
$$
p^*(K_{W}+\Delta)=p'^*(K_{W'}+\Delta')
$$
as $K_{W}+\Delta\equiv_X 0$.
Hence 
\begin{align*}
a(E_1; W')={}&\mult_{E_1^Z}(K_Z-p'^*K_{W'})+1\\
={}&\mult_{E_1^Z}(K_Z-p^*K_W-F)+1\\
={}&-\mult_{E_1}p_*F+1<1.
\end{align*}
Here $E_1^Z$ is the strict transform of $E_1$ on $Z$.
Take any exceptional prime divisor $E\neq E_1$ over $W'$, then $E$ is also exceptional over  $W$, and hence
\[a(E; W')\geq a(E; W', \Delta')=a(E; W, \Delta)>1.\]
Note that $W'$ is $\bQ$-factorial, so we conclude that $W'$ is  extremely non-canonical.  The fact that $\mld(X)\leq \mld(W')$ follows easily from
$\mld(W')= \text{ld}(W')\geq  \text{ld}(W', \Delta')=\mld(X).$
\end{proof}

\section{The $1$-gap theorem for $3$-dimensional extremely non-canonical singularities: the hyperquotient case}\label{section HQ}

In this section, we treat a special case of Theorem \ref{gap main}, where $X$ is an isolated extremely non-canonical singularity whose index $1$ cover is an isolated cDV singularity.
This is the most technical part of this paper. In the proof, we mimic the classification of $3$-dimensional terminal singularities following the explanation given by Reid \cite[Sections 6 and 7]{YPG} case by case. Of course our situation is more complicated than the case of terminal singularities, but the strategy of \cite[Sections 6 and 7]{YPG} still works after some modifications. The essential differences in our proof are that we replace the criterion for a hyperquotient singularity to be terminal (\cite[(4.6) Theorem]{YPG}) by our new criterion for a hyperquotient singularity to be extremely non-canonical (Theorem \ref{key thm}), which leads to more non-trivial discussions in each case; and in order to apply the terminal lemma as in \cite[Section 7]{YPG}, we need to first apply our new ``non-canonical" lemma (Lemma \ref{non-canonical lemma}) to exclude certain cases to guarantee the condition of the terminal lemma.
We try to write down all the details to make the proof convincible and friendly to readers not familiar with \cite{YPG}.

The following is the main theorem of this section.

\begin{thm}\label{no hyperquotient}There exists a positive real number $\delta>0$ such that there is no $3$-dimensional hyperquotient singularity $(P\in X)=(Q\in Y)/\bmu_r$ satisfying the followings: 
\begin{enumerate}
\item $(Q\in Y)$ is the canonical index $1$ cover of $(P\in X)$;
\item $(Q\in Y)\subset \bA^4$ is an isolated cDV singularity;
\item $(P\in X)$ is an isolated extremely non-canonical singularity with $\mld(X)>1-\delta$.
\end{enumerate}
In fact, we can take $\delta=\delta_0$, where $\delta_0$ is the constant from Lemma \ref{non-canonical lemma}.
\end{thm}

\begin{proof}[Outline of the proof]
To the contrary, assume that such a $3$-dimensional hyperquotient singularity $(P\in X)=(Q\in Y)/\bmu_r$ exists. Suppose that $Y=(f=0)\subset \bA^4$ and the hyperquotient is of type $\frac{1}{r}(a,b,c,d;e)$. 

In Subsection \ref{sub setting rules}, we introduce basic settings and restrictions on $f$ and $\frac{1}{r}(a,b,c,d;e)$, and roughly splits the possible $f$ into $5$ cases: $cA$, odd, $cD_4$, $cD_n$, $cE$.

In Subsection \ref{check terminal lemma}, using Lemma \ref{non-canonical lemma}, we check that $\frac{1}{r}(a,b,c,d;e,1)$ satisfies the assumption of the terminal lemma (Theorem \ref{terminal lemma}).

By applying the terminal lemma, we can get all possible values for $\frac{1}{r}(a,b,c,d;e)$ in each case.
In Subsection \ref{cA}, we exclude the $cA$ case.
In Subsection \ref{odd}, we exclude the odd case.
In Subsection \ref{cDE}, we exclude the $cD_4$, $cD_n$, $cE$ cases.
Then the nonexistence is proved.
\end{proof}
\subsection{Settings and rules}\label{sub setting rules}

In this subsection we introduce the settings and rules. 

Throughout the remaining part of this section, 
we take $\delta=\delta_0$, where $\delta_0$ is the constant from Lemma \ref{non-canonical lemma}. Recall that $\delta\leq \delta_3$, where $\delta_3$ is the constant from Corollary \ref{acc quot 3}.

We assume that such a $3$-dimensional hyperquotient singularity $(P\in X)=(Q\in Y)/\bmu_r$ as in Theorem \ref{no hyperquotient} exists, and we will exclude all the possibilities to get a contradiction.

As the index of $X$ is $r$ and $\mld(X)<1$, it is obvious that $\mld(X)\leq 1-\frac{1}{r}$.
Since $\delta\leq \delta_3\leq \frac{1}{13}$ by Example \ref{12/13}, we always have $r>13$.

We will freely and frequently use the notation in Subsection \ref{subsec cyclic quotient}.
Set $(x,y,z,t)=(x_1, x_2, x_3, x_4)$ to be the local analytic coordinates on $\bA^4$, $Y=(f=0)\subset \bA^4$, and the action of $\bmu_r$ is given by 
\[
\bmu_r\ni \xi: (x, y, z, t; f)\mapsto (\xi^{a}x, \xi^{b}y, \xi^{c}z, \xi^{d}t;\xi^{e}f).
\]
We also identify $(a,b,c,d)=(a_1, a_2, a_3, a_4)$, and recall that these weights are viewed as elements in $\ZZ/r$. Note that all monomials in $f$ shall have the same weight $e \bmod r$ as $f( \xi x, \xi y,  \xi z,  \xi t)= \xi^ef(x,y,z,t)$ by the setting.

We will always assume the following rules in the proof, which are similar to that of \cite[(6.6)]{YPG} except that Rule I is changed according to our assumption:

\medskip

\noindent {\bf Rule I:} (i) There exists at most one primitive vector $\beta\in N\cap \sigma$ such that 
\[
1-\delta< \beta(x_1\cdots x_4)-\beta(f)<1;
\]

(ii) for any vector $\alpha\in N\cap \sigma$ such that $\alpha\neq \beta$,
\[
 \alpha(x_1\cdots x_4)-\alpha(f)>1.
\]

\medskip

\noindent {\bf Rule II:} (i) If $\gcd(a_i, r)\neq 1$, then $a_i$ divides $e$, that is, $\gcd(a_i, r)$ divides $\gcd(e, r)$; 

(ii) $\gcd(a_i, a_j, r)=1$ for all $i\neq j$;

(iii) $a+b+c+d-e=1.$

\medskip

\noindent {\bf Rule III:} 
(i)  After a $\bmu_r$-equivariant analytic change of coordinates, we may assume that $f=q(x_1,\ldots , x_k)+f'(x_{k+1},\ldots, x_4)$ with $q$ a nondegenerate quadratic form in $x_1,\ldots , x_k$;

(ii) if the $3$-jet of $f$ is $x^2+y^2z$ then 
\[
f=x^2+y^2z+yg(t)+h(z,t),
\]
or if the $3$-jet is $x^2+y^3$ then 
\[
f=x^2+y^3+yg(z, t)+h(z,t).
\]

\medskip

Here Rule I is the conclusion of Theorem \ref{key thm}. Rules II and III are exactly the same with that in \cite[Page 394]{YPG}. As explained in \cite{YPG}, Rule II(i)(ii) are consequences of the fact that 
$\bmu_r$ acts freely on $Y$ outside $Q$; Rule II(iii) comes from the following: $\bmu_r$ acts on the generator $s\in \omega_Y$ by $\bmu_r\ni \xi: s\mapsto \xi^{a+b+c+d-e} s$, and the index of $K_X$ is $r$, which means that $a+b+c+d-e$ is coprime to $r$, so we may assume that $a+b+c+d-e=1$ by changing the choice of primitive root; Rule III is standard in singularity theory by taking analytic change of coordinates (see \cite[Page 394--395]{YPG}).

In fact, by Rule III, we can divide the possible $f$ into 5 cases by \cite[(6.7)]{YPG} as the following:
\begin{prop}[{\cite[(6.7) Proposition]{YPG}}]\label{f cases}
By making a $\bmu_r$-equivariant analytic change of coordinates and possibly permuting the coordinates, $f$ can be only in the following $5$ cases:

$cA$ case: $f=xy+g(z,t)$ with $g\in \mm^2$;

odd case: $f=x^2+y^2+g(z,t)$ with $g\in \mm^3$ and $a\not\equiv b \bmod r$;

$cD_4$ case: $f=x^2+g(y,z,t)$ with $g\in \mm^3$ and $g_3$ is a reduced cubic;

$cD_n$ case: $f=x^2+y^2z+g(z,t)$ with $g\in \mm^4$;

$cE$ case: $f=x^2+y^3+yg(z,t)+h(z,t)$ with $g\in \mm^3$ and $h\in \mm^4$.

Here $\mm$ is the maximal ideal of $\bC[{x,y,z,t}]$ and $g_3$ is the cubic part of $g$.
\end{prop}
For the proof we refer to that in \cite{YPG} and we remark that the proof only uses Rule III.

\subsection{Reduction to the terminal lemma}\label{check terminal lemma}
In this subsection, we check that $\frac{1}{r}(a,b,c,d;e,1)$ satisfies the assumption of the terminal lemma (Theorem \ref{terminal lemma}), similar to \cite[(7.2)]{YPG}. But in our setting
the existence of $\beta\in N\cap \sigma$ makes the situation more complicated. Usually $\beta$ and  $\beta'=(1,\ldots, 1)-\beta$ should be considered separately from other vectors. Note that the coprimeness is not treated in \cite[(7.2)]{YPG} but later case by case, while in our situation, we should check the coprimeness in the middle of the proof before dealing with $\beta$. This is because we need to apply Lemma \ref{non-canonical lemma} to exclude certain cases of $\beta$, where the coprimeness is already needed.

We denote by $\square$  the unit cube of $N_\bR$ removing 16 vertices, i.e., $\square=[0,1]^4\setminus \{0,1\}^4\subset \bR^4$.
For any $\alpha \in N\cap \square$, we will always use $\alpha'$ to denote the vector $\alpha'=(1,\ldots, 1)-\alpha$.
For $k=1, \ldots, r-1$, denote $\alpha_k=\frac{1}{r}(\overline{ak},\overline{bk},\overline{ck},\overline{dk})\in N\cap \square$, where $\overline{n}$ is the smallest non-negative residue modulo $r$. Note that $\alpha'_k=\alpha_{r-k}$ if and only if none of $\overline{ak},\overline{bk},\overline{ck},\overline{dk}$ is $0$.  Also note that for any $\alpha \in N\cap \square$, there exists 
some $k=1, \ldots, r-1$, such that  $\alpha\equiv \alpha_k \bmod \ZZ^4$, and  $\alpha= \alpha_k$ holds if and only if none of $\alpha(x), \alpha(y), \alpha(z), \alpha(t)$ is $1$.

\begin{prop}\label{prop xy}
Suppose that $xy\in f$. Then the followings hold.

\begin{enumerate}
\item For any $\alpha \in N\cap \square$ such that $\alpha\neq \beta, \beta'$, one of the followings holds:
\begin{enumerate}
\item[(i)] $\alpha(f)=\alpha(xy)\leq 1$ and $\alpha(zt)>1$, moreover, if $\alpha(xy)=1$, then one of $\alpha(z), \alpha(t)$ is $1$;
\item[(ii)] $\alpha(f)=\alpha(xy)-1$ and $\alpha(zt)<1$,  moreover, if $\alpha(xy)=1$, then one of $\alpha(z), \alpha(t)$ is $0$.
\end{enumerate}
The alternative cases are interchanged by the symmetry $\alpha\mapsto \alpha'=(1,\ldots, 1)-\alpha$.
 In particular,  for $k=1, \ldots, r-1$, if $\alpha_k\neq \beta, \beta'$, then these two cases imply
 \[
 \overline{ak}+\overline{bk}=\overline{ek} \text{ and } \overline{ck}+\overline{dk}=k+r
 \]
 or
 \[
 \overline{ak}+\overline{bk}=\overline{ek}+r \text{ and } \overline{ck}+\overline{dk}=k
 \]
 respectively.
 
 \item Denote $q=\gcd(e, r)$. Then $q=\gcd(d,r)$, and $a,b,c$ are coprime to $r$ (after possibly interchanging $z$ and $t$).
 \item If $\beta \in N\cap \square$, then there exists an integer $1\leq k_0\leq r-1$ such that $\beta=\alpha_{k_0}$. Moreover, $\beta(xy)\geq 1$ and $1-\delta<\beta(zt)<1$. In particular, 
 $1-\delta<\frac{k_0}{r}<1$ and
 \[
 \overline{ak_0}+\overline{bk_0}=\overline{ek_0}+r \text{ and } \overline{ck_0}+\overline{dk_0}=k_0.
 \]
 
 \item For any $k=1, \ldots, r-1$,
 \[
 \overline{ak}+\overline{bk}+\overline{ck}+\overline{dk}=\overline{ek}+k+r. 
 \]
\end{enumerate}
\end{prop}

\begin{proof}
(1) As $xy\in f$, $a+b\equiv e$, and $c+d\equiv 1 \bmod r$ by Rule II(iii). 
Since $a+b\equiv e \bmod r$, it is easy to see that $a$ and $b$ are coprime to $r$ by Rule II(i)(ii).

By  $a+b\equiv e \bmod r$, $\alpha(f)\equiv \alpha(xy) \bmod \ZZ$ for all $\alpha \in N\cap \sigma$. Fix any $\alpha \in N\cap \square$. Since $0\leq \alpha(f)\leq \alpha(xy)< 2$, either $\alpha(f)=\alpha(xy)$ or $\alpha(f)=\alpha(xy)-1$. 
Note that $\alpha(xy)< 2$ because otherwise $\alpha(x)=\alpha(y)=1$, which contradicts the fact that $\gcd(a,b,r)=1$ by Rule II(ii). By Rule I, if $\alpha(f)=\alpha(xy)$ and $\alpha \neq \beta$, then $\alpha(zt)>1$.

Suppose that $\alpha \neq \beta , \beta'$, certainly $\alpha' \neq \beta , \beta'$. There are two cases: (i) $\alpha(f)=\alpha(xy)$; (ii) $\alpha(f)=\alpha(xy)-1$.

\medskip

\noindent{\bf Case (ii)}: Assume that $\alpha(f)=\alpha(xy)-1$, then $\alpha(xy)\geq 1$.

If $ \alpha(xy)=1$, then $\alpha(f)=0$ and there is a monomial in $f$ with weight $0$.  None of $\alpha(x), \alpha(y)$ is $0$ since $a$ and $b$ are coprime to $r$, so one of $\alpha(z), \alpha(t)$ is $0$, and in this case $\alpha(zt)<1$ holds. 

If $ \alpha(xy)>1$, then  $\alpha'(xy)< 1$, and  hence  $\alpha'(f)=\alpha'(xy)$. This implies that $\alpha'(zt)>1$ by Rule I and hence $\alpha(zt)<1$. This proves (ii).

\medskip

\noindent{\bf Case (i)}: Assume that $\alpha(f)=\alpha(xy)$, then $\alpha(zt)>1$ by Rule I. 
Suppose that $\alpha(xy)>1$, then $\alpha'(xy)<1$ which implies that $\alpha'(f)=\alpha'(xy)$ and $\alpha'(zt)>1$, which contradicts $\alpha(zt)>1$. Hence  $\alpha(xy)\leq 1$. 
On the other hand, if $\alpha(xy)=1$, then the same argument implies that $\alpha'(xy)=1$ and $\alpha'(f)=0$. By case (ii), one of $\alpha'(z), \alpha'(t)$ is $0$, which implies that  one of $\alpha(z), \alpha(t)$ is $1$. This proves (i).

\medskip

Therefore, the former part of statement (1) is proved. Note that $\alpha(zt)<1$ if and only if $\alpha'(zt)>1$, so the alternative cases are interchanged by the symmetry.

For the latter part, note that $\alpha_k(xy)=\frac{1}{r} (\overline{ak}+\overline{bk})\equiv \frac{1}{r}\overline{ek} \bmod \ZZ$ and 
$\alpha_k(zt)=\frac{1}{r}(\overline{ck}+\overline{dk})\equiv \frac{k}{r} \bmod \ZZ$.
If $\alpha_k$ is in case (i), then $\alpha_k(xy)\neq 1$ since  $\alpha_k(z), \alpha_k(t)<1$, therefore $\alpha_k(xy)<1$ and $1<\alpha_k(zt)<2$, which gives the first equation. If $\alpha_k$ is in case (ii), then $1\leq \alpha_k(xy)<2$  and $\alpha_k(zt)<1$, which gives the second equation.


\bigskip

Before proving (2), we note that if $\beta\in N\cap \square$, then $\beta(f)=\beta(xy)$, because otherwise $\beta(f)=\beta(xy)-1$, and $\beta(zt)<0$ by Rule I, which is absurd. It follows that $1-\delta<\beta(zt)<1$ by Rule I. 

\bigskip

(2) Since $a+b\equiv e \bmod r$, it is easy to see that $a$ and $b$ are coprime to $r$ by Rule II(i)(ii). By Rule II(i), $\gcd(c,r)$ and $\gcd(d,r)$ divide $q=\gcd(e,r)$. So by Rule II(ii), it suffices to show that $q$ divides either $c$ or $d$. We may assume that $q>1$, and set $k_1=r/q \leq r/2$. If either $\overline{ck_1}=0$ or $\overline{dk_1}= 0$, then $q$ divides 
either $c$ or $d$, and we are done. So we may assume that $\overline{ck_1}\neq 0, \overline{dk_1}\neq 0$ and try to get a contradiction. In particular this means that $\alpha_{r-k_1}=\alpha'_{k_1}$. We need to consider 3 cases: $\alpha_{k_1}\neq \beta, \beta'$; $\alpha_{k_1}= \beta$; $\alpha_{k_1}= \beta'$.

If $\alpha_{k_1}\neq \beta, \beta'$, then also $\alpha_{r-k_1}\neq \beta, \beta'$. Hence by $\overline{ek_1}=\overline{e(r-k_1)}=0$, we are in the second case of (1), that is,
\[
 \overline{ck_1}+\overline{dk_1}=k_1 \text{ and } \overline{c(r-k_1)}+\overline{d(r-k_1)}=r-k_1,
\]
but this is absurd, since the sum of the left hand sides of the equations above should be $2r$ as $\overline{ck_1}\neq 0, \overline{dk_1}\neq 0$.

If $\alpha_{k_1}= \beta$, then $1-\delta<\beta(zt)<1$ implies that 
\[
1-\delta<\frac{\overline{ck_1}+\overline{dk_1}}{r}=\frac{k_1}{r}<1.
\]
But this contradicts $k_1\leq r/2$.

If $\alpha_{k_1}= \beta'$, then $\alpha_{r-k_1}= \beta$, and $1-\delta<\beta(zt)<1$ implies that 
\[
1-\delta<\frac{\overline{c(r-k_1)}+\overline{d(r-k_1)}}{r}=\frac{r-k_1}{r}<1.
\]
This means that $k_1/r<\delta\leq \frac{1}{6}$ and hence $q>6$. For $j=2,3,5$, we may consider $jk_1=jr/q<r$ and consider the weighting $\alpha_{jk_1}$. Note that by the construction, $k_1<jk_1<r-k_1$, hence $\alpha_{jk_1}\neq \beta, \beta'$ for $j=2,3,5$ (same holds for $\alpha_{r-jk_1}$). Hence by $\overline{ejk_1}=0$, we are in the second case of (1), that is,
\[
 \overline{cjk_1}+\overline{djk_1}= jk_1 \text{ and } \overline{c(r-jk_1)}+\overline{d(r-jk_1)}=r- jk_1.
\]
Since the sum of right hand sides of the equations above  is $r$, either $\overline{cjk_1}= 0$ or $\overline{djk_1}=0$ for each $j=2,3,5$. After possibly interchanging $z,t$, we may assume that $\overline{djk_1}= 0$ holds for at least two $j\in \{2,3,5\}$, but this implies that $\overline{dk_1}= 0$, a contradiction.

\bigskip

(3) Now suppose that $\beta\in N\cap \square$. Recall that 
$\beta(f)=\beta(xy)$ and $1-\delta<\beta(zt)<1$.

First we show that there exists an integer $1\leq k_0\leq r-1$ such that $\beta=\alpha_{k_0}$.
Note that by definition there exists an integer $1\leq k_0\leq r-1$ such that $\beta\equiv\alpha_{k_0} \bmod \ZZ$. 
Since $a, b$ are coprime to $r$ by (2), $\beta(x),\beta(y)$ are not $1$, which means that $\beta=\alpha_{k_0}$. 
Note that in this case 
\[
1-\delta<\beta(zt)=\frac{1}{r}(\overline{ck_0}+\overline{dk_0})=\frac{k_0}{r}<1.
\]

Then we will show that $\beta(xy)\geq 1$. 
Suppose that $\beta(xy)<1$, then we know that 
$
 \overline{ak_0}+\overline{bk_0}=\overline{ek_0}.$ Hence 
 \[
 \overline{ak_0}+\overline{bk_0}+\overline{ck_0}+\overline{dk_0}=\overline{ek_0}+k_0.
 \]
On the other hand, for any $1\leq k\leq r-1$ such that $k\neq k_0$, if $\alpha_k\neq \beta'$, then by (1), 
 \[
 \overline{ak}+\overline{bk}+\overline{ck}+\overline{dk}=\overline{ek}+r+k;
 \]
if $\alpha_k= \beta'$, then 
 \[
 \overline{ak}+\overline{bk}+\overline{ck}+\overline{dk}=4r-\overline{ek_0}-k_0\geq 2r.
 \]
Hence $\frac{1}{r}(a,b,c,d;e)$ and $k_0$ satisfy the assumption of Lemma \ref{non-canonical lemma},  which implies that $k_0/r\leq  1-\delta$, a contradiction. 
Here 
the coprimeness follows from (2) and (\ding{73}1) of Lemma \ref{non-canonical lemma} is satisfied.
This concludes (3).

\bigskip

(4) For any $k=1, \ldots, r-1$, if $\alpha_k\neq \beta'$, then the statement follows from (1) and (3). 
 If $\alpha_k= \beta'=\alpha'_{k_0}$, then $k=r-k_0$; also we know that $\overline{ek_0}\neq 0$, because otherwise $\overline{dk_0}= 0$ by (2), which contradicts $\alpha_k=\alpha'_{k_0}$; therefore,
\[
 \overline{ak}+\overline{bk}+\overline{ck}+\overline{dk}=4r-\overline{ek_0}-k_0-r=\overline{e(r-k_0)}+r-k_0+r.
 \]
 This concludes (4).
\end{proof}

\begin{prop}\label{prop x2}
Suppose that $x^2\in f$. Then the followings hold.
\begin{enumerate}
\item For any $\alpha \in N\cap \square$ such that $\alpha\neq \beta, \beta'$, 
one of the followings holds:
\begin{enumerate}
\item[(i)] $\alpha(f)=2\alpha(x)\leq 1$ and $\alpha(yzt)>1+\alpha(x)$, moreover, if $2\alpha(x)=1$, then $(\alpha(y), \alpha(z), \alpha(t))$ is a permutation of $(\frac{1}{2}, \frac{1}{2}, 1)$;
\item[(ii)] $\alpha(f)=2\alpha(x)-1$ and $\alpha(yzt)<1+\alpha(x)$,  moreover,  if $2\alpha(x)=1$, then $(\alpha(y), \alpha(z), \alpha(t))$ is a permutation of $(\frac{1}{2}, \frac{1}{2}, 0)$.
\end{enumerate}
The alternative cases are interchanged by the symmetry $\alpha\mapsto \alpha'=(1,\ldots, 1)-\alpha$.
 In particular, for $k=1, \ldots, r-1$, if $\alpha_k\neq \beta, \beta'$, then these two cases imply
 \[
2 \overline{ak}=\overline{ek} \text{ and } \overline{bk}+\overline{ck}+\overline{dk}=\overline{ak}+k+r
 \]
 or
 \[
2 \overline{ak}=\overline{ek}+r \text{ and } \overline{bk}+\overline{ck}+\overline{dk}=\overline{ak}+k
 \]
 respectively.
 
 \item One of the followings holds (after possibly interchanging $y, z,t$):
 \begin{enumerate}
 \item $a\equiv e\equiv 0 \bmod r$ and $b,c,d$ are coprime to $r$;
 \item $r$ is odd and $a,b,c,d,e$ are coprime to $r$;
 \item $q=2=\gcd(d,r)=\gcd(e,r)$ and $a, b,c$ are coprime to $r$.
 \end{enumerate}
 
 \item If $\beta \in N\cap \square$, then there exists an integer $1\leq k_0\leq r-1$ such that $\beta\equiv \alpha_{k_0} \bmod \ZZ^4$. 
 Moreover, $1-\delta<\frac{k_0}{r}<1$ and
 \[
 \overline{ak_0}+\overline{bk_0}+\overline{ck_0}+\overline{dk_0}=\overline{ek_0}+k_0+r. 
 \]

 \item For any $k=1, \ldots, r-1$,
 \[
 \overline{ak}+\overline{bk}+\overline{ck}+\overline{dk}=\overline{ek}+k+r. 
 \]
\end{enumerate}

\end{prop}

\begin{proof}
(1) As $x^2\in f$, $2a\equiv e \bmod r$, and $b+c+d\equiv 1+a \bmod r$ by Rule II(iii). From the former one, $\alpha(f)\equiv 2\alpha(x) \bmod \ZZ$ for all $\alpha \in N\cap \sigma$. Fix any $\alpha \in N\cap \square$. Since $0\leq \alpha(f)\leq 2\alpha(x)\leq 2$, either $\alpha(f)=2\alpha(x)$ or $\alpha(f)=2\alpha(x)-1$. Note that here $\alpha(f)=2\alpha(x)-2$ is impossible since otherwise $\alpha(f)=0$ and $\alpha(x)=1$, but $\alpha(f)=0$ implies that at least one of $\alpha(y), \alpha(z), \alpha(t)$ is 0, and hence $x$ and one of $y,z,t$ have a common factor with $r$, which contradicts Rule II(ii).
By Rule I, if $\alpha(f)=2\alpha(x)$ and $\alpha \neq \beta$, then $\alpha(yzt)>1+\alpha(x)$.

Suppose that $\alpha \neq \beta , \beta'$, then $\alpha' \neq \beta , \beta'$. There are two cases: 
(i) $\alpha(f)=2\alpha(x)$; (ii) $\alpha(f)=2\alpha(x)-1$.

\medskip

\noindent{\bf Case (ii)}: Assume that $\alpha(f)=2\alpha(x)-1$, then $2\alpha(x)\geq 1$.

If $2\alpha(x)= 1$, then $\alpha(f)=0$ and hence one of $\alpha(y), \alpha(z), \alpha(t)$ is 0, say $\alpha(t)$. Recall that there exists an integer $1\leq k\leq r-1$ such that $\alpha\equiv \alpha_k \bmod \ZZ^4$. Then  $2\alpha(x)=1$ and $\alpha(t)=0$ implies that $\overline{2ak}=\overline{dk}=0$. Since $\gcd(a, d, r)=1$, this implies that $d$ and $r$ are even and $k=\frac{r}{2}$.
Since $\gcd(b, d, r)=\gcd(c, d, r)=1$,  $\overline{bk}=\overline{ck}=\frac{1}{2}$. This means that $(\alpha(y),\alpha(z), \alpha(t))=(\frac{1}{2}, \frac{1}{2}, 0)$, and in this case  $\alpha(yzt)<1+\alpha(x)$ holds. 

If $2\alpha(x)> 1$, then $2\alpha'(x)< 1$ and hence $\alpha'(f)=2\alpha'(x)$. This implies that  $\alpha'(yzt)>1+\alpha'(x)$ by Rule I and hence $\alpha(yzt)<1+\alpha(x)$. This proves (ii).

\medskip 

\noindent{\bf Case (i)}: Assume that $\alpha(f)=2\alpha(x)$, then 
 $\alpha(yzt)>1+\alpha(x)$ by Rule I.
 Suppose that $2\alpha(x)>1$, then $2\alpha'(x)<1$ which implies that $\alpha'(f)=2\alpha'(x)$, and $\alpha'(yzt)>1+\alpha'(x)$ by Rule I, but this contradicts  $\alpha(yzt)>1+\alpha(x)$. Hence $2\alpha(x)\leq 1$. On the other hand, if $2\alpha(x)=1$, then the same argument implies that 
$\alpha'(f)=0$ and $2\alpha'(x)=1$. By case (ii), $(\alpha'(y),\alpha'(z), \alpha'(t))=(\frac{1}{2}, \frac{1}{2}, 0)$ after permutation, and this proves (i).

\medskip


Hence the former part of statement (1) is proved. Note that $\alpha(yzt)>1+\alpha(x)$ if and only if $\alpha'(yzt)<1+\alpha'(x)$, so the alternative cases are interchanged by the symmetry. 

The latter part follows easily by the fact that $2\alpha_k(x)=\frac{1}{r} (2\overline{ak})\equiv \frac{1}{r}\overline{ek} \bmod \ZZ$ and 
$\alpha_k(yzt)=\frac{1}{r}(\overline{bk}+\overline{ck}+\overline{dk})\equiv \frac{1}{r}(k+\overline{ak}) \bmod \ZZ$. To be more precise, if $\alpha_k$ is in case (i), we get that 
$2\alpha_k(x)=\frac{\overline{ek}}{r}<1$, and $\alpha_k(yzt)= \frac{1}{r}(k+\overline{ak})+1$ or $\frac{1}{r}(k+\overline{ak})+2$. We need to show that $\alpha_k(yzt)= \frac{1}{r}(k+\overline{ak})+2$ can not happen. Suppose that it happens, then $\alpha'_k(yzt)=1- \frac{1}{r}(k+\overline{ak})$, $\alpha'_k(x)=1-\frac{\overline{ak}}{r}$, and $\alpha'_k(f)=2\alpha'_k(x)-1=1-\frac{2\overline{ak}}{r}$, but this contradicts Rule I.
So this case we get the first equation. 
 If $\alpha_k$ is in case (ii), we get that 
$2>2\alpha_k(x)\geq 1$, and $\alpha_k(yzt)= \frac{1}{r}(k+\overline{ak})$ or $\frac{1}{r}(k+\overline{ak})-1$. Hence $2\alpha_k(x)=\frac{\overline{ek}}{r}+1$. Here note that $\alpha_k(yzt)= \frac{1}{r}(k+\overline{ak})-1$ can not happen because it contradicts the fact that
$\alpha_k(f)=\frac{2\overline{ak}}{r}-1$ and Rule I. So this case we get the second equation.

\bigskip

Note that if $\beta\in N\cap \square$, then we also have either $\beta(f)=2\beta(x)$ and $1-\delta<\beta(yzt)-\beta(x)<1$; or $\beta(f)=2\beta(x)-1$, and $-\delta<\beta(yzt)-\beta(x)<0$ by Rule I. In particular, if $\beta\equiv \alpha_k \bmod \ZZ^4$, then $\beta(yzt)-\beta(x)\equiv \frac{k}{r}\bmod \ZZ$, which implies that $1-\delta< \frac{k}{r}<1$ in both cases.

\bigskip

(2) First we show that if $\gcd(a,r)\neq 1$, then $\gcd(e,r)=\gcd(a,r)$ and $b, c, d$ are coprime to $r$. 
Suppose that $\gcd(a,r)=q_1> 1$ and $\gcd(e,r)\neq \gcd(a,r)$, then since $e\equiv 2a\bmod r$, we know that $\gcd(e,r)=2q_1$. In particular, $r$ is even. Take $k_1=\frac{r}{2q_1}$, then $\overline{ek_1}=0$ and $\overline{ak_1}=\overline{a(r-k_1)}=\frac{r}{2}$. Note that by Rule II(ii), $q_1$ does not divide $b,c,d$, hence $\overline{bk_1}, \overline{ck_1}, \overline{dk_1}$ can not be $0$ or $\frac{r}{2}$. In particular, $\alpha'_{k_1}=\alpha_{r-k_1}$.
We need to consider $3$ cases: $\alpha_{k_1}\neq \beta, \beta'$; $\alpha_{k_1}= \beta$; $\alpha_{k_1}= \beta'$.

If $\alpha_{k_1}\neq \beta, \beta'$, then $\alpha_{r-k_1}\neq \beta, \beta'$, and we are in the second case of (1), which gives 
\[
 \overline{bk_1}+\overline{ck_1}+\overline{dk_1}=\frac{r}{2}+k_1
\]
and
\[
 \overline{b(r-k_1)}+\overline{c(r-k_1)}+\overline{d(r-k_1)}=\frac{r}{2}+r-k_1.
\]
But this is absurd since the sum of the left hand sides of the equations above  is $3r$.

If $\alpha_{k_1}= \beta$, 
we get $1-\delta <\frac{k_1}{r}<1$. But this implies that $\delta>\frac{3}{4}$ since $k_1\leq \frac{r}{4}$, which is a contradiction.

If $\alpha_{k_1}= \beta'$, then $\alpha_{r-k_1}=\beta$, and 
$1-\delta <\frac{r-k_1}{r}<1$.
In particular, $\frac{k_1}{r}<\delta\leq \frac{1}{12}$ and hence $2q_1>12$. For $j=3,5,7,11$, we may consider $jk_1=\frac{jr}{2q_1}<r$ and consider the weighting $\alpha_{jk_1}$. Note that by the construction, $k_1<jk_1<r-k_1$, hence $\alpha_{jk_1}\neq \beta,  \beta'$ for $j=3,5,7,11$ (same holds for $\alpha_{r-jk_1}$). Hence by $\overline{ajk_1}=\frac{r}{2}$ and $\overline{ejk_1}=0$, we are in the second case of (1), that is,
\[
 \overline{bjk_1}+\overline{cjk_1}+\overline{djk_1}=\frac{r}{2}+jk_1
 \]
 and 
 \[
 \overline{b(r-jk_1)}+\overline{c(r-jk_1)}+\overline{d(r-jk_1)}=\frac{r}{2}+r-jk_1.
\]
Since the sum of the right hand sides of the equations above  is $2r$, one of $\overline{bjk_1}$, $\overline{cjk_1}$, $\overline{djk_1}$ is 0 for each $j=3,5,7,11$. After possibly interchanging $y,z,t$, we may assume that $\overline{bjk_1}= 0$ for at least two $j\in \{3,5,7,11\}$. But as these two $j$'s are coprime, this implies that $\overline{bk_1}= 0$, a contradiction.

Therefore we showed that if $\gcd(a,r)\neq 1$, then $\gcd(e,r)=\gcd(a,r)$. In this case $b, c, d$ are coprime to $r$ by Rule II(i)(ii).

Now assume that $\gcd(a,r)= 1$. Since $e\equiv 2a \bmod r$, it follows that $\gcd(e, r)=1$ if $r$ is odd, and $\gcd(e, r)=2$ if $r$ is even. Moreover, in the first case, $a, b,c,d$ are coprime to $r$ by Rule II(i). 

Now consider $r$ is even and $\gcd(e, r)=2$. Note that at most one of $b,c,d$ is even by Rule II(ii). 
Suppose that none of them is even.
Set $k_2=\frac{r}{2}$. Then $\overline{ek_2}=0$ and $\overline{ak_2}=\overline{bk_2}=\overline{ck_2}=\overline{dk_2}=\frac{r}{2}$. That is, $\alpha_{k_2}=(\frac{1}{2},\frac{1}{2},\frac{1}{2},\frac{1}{2})$. Note that $\alpha_{k_2}\neq \beta, \beta'$, otherwise $\beta=\beta'$, $\beta(xyzt)=2$, and $2\beta(f)\in \ZZ$, which contradicts Rule I and $\delta<\frac{1}{2}$. 
Then we get a contradiction since $\alpha_{k_2}(yzt)=\alpha_{k_2}(x)+1$, which violates both situations in (1).
Hence exactly one of $b,c,d$ is even, say $d$, and in this case $\gcd(d, r)=\gcd(e, r)=2$.

In summary, we showed that one of the followings holds:
\begin{enumerate}
 \item[(a')] $\gcd(e,r)=\gcd(a,r)\neq 1$ and $b, c, d$ are coprime to $r$;
 \item[(b)] $r$ is odd and $a,b,c,d,e$ are coprime to $r$;
 \item[(c)] $q=2=\gcd(d,r)=\gcd(e,r)$ and $a, b,c$ are coprime to $r$.
 \end{enumerate}
 To conclude statement (2), we only need to prove that if $\gcd(a,r)\neq 1$ then $a\equiv 0 \bmod r$. But we will come back to this after proving (3) and (4).

\bigskip

(3) Now suppose that $\beta\in N\cap \square$. Recall that either $\beta(f)=2\beta(x)$ and $1-\delta<\beta(yzt)-\beta(x)<1$; or $\beta(f)=2\beta(x)-1$, and $-\delta<\beta(yzt)-\beta(x)<0$ by Rule I.

Note that there exists an integer $1\leq k_0\leq r-1$ such that $\beta\equiv \alpha_{k_0} \bmod \ZZ$.

If $\alpha_{k_0}\neq \beta, \beta'$, then we get the desired equality by (1).
If $\alpha_{k_0}=\beta'$, then $\beta\equiv \beta' \bmod \ZZ^4$, which implies that $\beta(x), \beta(y), \beta(z), \beta(t)\in \{0, \frac{1}{2}, 1\}$, and $2(\beta(yzt)-\beta(x))$ is an integer, but contradicts   $\delta<\frac{1}{2}$. 

Hence we may assume that $\alpha_{k_0}=\beta$. Recall that $\frac{k_0}{r}\equiv \beta(yzt)-\beta(x) \bmod \ZZ$ and $1-\delta <\frac{k_0}{r}<1$.

Suppose that $\beta(f)=2\beta(x)-1$, and $-\delta<\beta(yzt)-\beta(x)<0$. Then $2\beta(x)\geq 1$ and this implies that 
\[
2 \overline{ak_0}=\overline{ek_0}+r \text{ and } \overline{bk_0}+\overline{ck_0}+\overline{dk_0}=\overline{ak_0}+k_0-r.
\]
This gives
 \[
 \overline{ak_0}+\overline{bk_0}+\overline{ck_0}+\overline{dk_0}=\overline{ek_0}+k_0.
 \]
On the other hand, for any $1\leq k\leq r-1$ such that $k\neq k_0$, if $\alpha_k\neq \beta'$, then by (1), 
 \[
 \overline{ak}+\overline{bk}+\overline{ck}+\overline{dk}=\overline{ek}+r+k;
 \]
if $\alpha_k= \beta'$, then 
 \[
 \overline{ak}+\overline{bk}+\overline{ck}+\overline{dk}=4r-\overline{ek_0}-k_0\geq 2r.
 \]
Hence $\frac{1}{r}(a,b,c,d;e)$ and $k_0$ satisfy the assumption of Lemma \ref{non-canonical lemma} after possibly relabeling $a,b,c,d$ properly, which implies that $k_0/r\leq 1-\delta$, a contradiction.  Here  
the coprimeness is guaranteed by one of (a')(b)(c) in the part of (2) we just proved; (\ding{73}2) of Lemma \ref{non-canonical lemma} is satisfied in case (a') by labeling $a=a_4$;  (\ding{73}3)  of Lemma \ref{non-canonical lemma} is satisfied in cases (b)(c).

Suppose that $\beta(f)=2\beta(x)$, and $1-\delta<\beta(yzt)-\beta(x)<1$. Note that in this case, $\beta(yzt)-\beta(x)=\frac{k_0}{r}$, and hence 
 \[
 \overline{bk_0}+\overline{ck_0}+\overline{dk_0}=\overline{ak_0}+k_0.
 \]
Then we will show that $2\beta(x)\geq 1$. 
Suppose that $2\beta(x)<1$, then we know that 
$
 2\overline{ak_0}=\overline{ek_0}$ and 
 \[
\overline{ak_0}+ \overline{bk_0}+\overline{ck_0}+\overline{dk_0}=\overline{ek_0}+k_0.
 \]
On the other hand, for any $1\leq k\leq r-1$ such that $k\neq k_0$, if $\alpha_k\neq \beta'$, then by (1), 
 \[
 \overline{ak}+\overline{bk}+\overline{ck}+\overline{dk}=\overline{ek}+r+k;
 \]
if $\alpha_k= \beta'$, then 
 \[
 \overline{ak}+\overline{bk}+\overline{ck}+\overline{dk}=4r-\overline{ek_0}-k_0\geq 2r.
 \]
Hence $\frac{1}{r}(a,b,c,d;e)$ and $k_0$ satisfy the assumption of Lemma \ref{non-canonical lemma} after possibly relabeling $a,b,c,d$ properly, which implies that $k_0/r\leq 1-\delta$, a contradiction. Here  
the coprimeness is guaranteed by one of (a')(b)(c) in the part of (2) we just proved; (\ding{73}2) of Lemma \ref{non-canonical lemma} is satisfied in case (a') by labeling $a=a_4$;  (\ding{73}3)  of Lemma \ref{non-canonical lemma} is satisfied in cases (b)(c). 

Hence $2\beta(x)\geq 1$ and we have 
$
 2\overline{ak_0}=\overline{ek_0}+r$. 
This concludes (3).

\bigskip

(4) For any $k=1, \ldots, r-1$, if $\alpha_k\neq \beta'$, then the statement follows from (1) and (3). 
 If $\alpha_k= \beta'=\alpha'_{k_0}$, then $k=r-k_0$; also we know that $\overline{ek_0}\neq 0$, because otherwise $\overline{ak_0}= 0$ or $\overline{dk_0}= 0$ by case (a') or (c) of (2), which contradicts $\alpha_k=\alpha'_{k_0}$; therefore,
\[
 \overline{ak}+\overline{bk}+\overline{ck}+\overline{dk}=4r-\overline{ek_0}-k_0-r=\overline{e(r-k_0)}+r-k_0+r.
 \]
 This concludes (4).
 
 \bigskip
 
 (2) (continued) Now we are ready to show that if $\gcd(a,r)\neq 1$ then $a\equiv 0 \bmod r$. If $\gcd(a,r)\neq 1$, then we are in case (a') of (2), and by (4), the assumption of Theorem \ref{terminal lemma} is satisfied. Hence we get $a\equiv e \bmod r$. On the other hand, $2a\equiv e \bmod r$. Hence $a\equiv e \equiv 0 \bmod r$.
\end{proof}

\begin{remark}
Before discussing case by case, we explain the strategy again. 
By Proposition \ref{prop xy}(2)(4) or \ref{prop x2}(2)(4), we checked that the $6$-tuple $\frac{1}{r}(a,b,c,d;e, 1)$ satisfies the terminal lemma (Theorem \ref{terminal lemma}). So  we can list all possible values for $\frac{1}{r}(a,b,c,d;e)$ in each case. Then we can apply Proposition \ref{prop xy}(1) or \ref{prop x2}(1) to some special $\alpha_{k_1}$ to get more restrictions on monomials in $f$, which leads to the final conclusion. For the smart choice of  $\alpha_{k_1}$ we just follow \cite[Section 7]{YPG}, but again the existence of $\beta$ gets in the way. So we have to consider the case that $\alpha_{k_1}=\beta$ or $\beta'$, in which we can not apply Proposition \ref{prop xy}(1) or \ref{prop x2}(1). In this case, we should consider to choose other special $\alpha_{k_2}$, $\alpha_{k_3}$, etc., and make more discussions.
\end{remark}

\subsection{The $cA$ case}\label{cA}
In this subsection, we consider case $cA$ in Proposition \ref{f cases}: $f=xy+g(z,t)$ with $g\in \mm^2$.

By Proposition \ref{prop xy}(2), $q=\gcd(d,r)=\gcd(e, r)$, and $c$ is coprime to $r$, this means that 
$q$ divides the degree of $z$ in each monomial in $g$, that is, we may write $g=g(z^q, t)$ by abusing the notation.

By Proposition \ref{prop xy}(2)(4), we can list all possible types by Theorem \ref{terminal lemma}, and one of the followings holds (after possibly interchanging $x,y$ or $z,t$):

If $q>1$,
\begin{enumerate}
\item[(A)] $a+b\equiv 0, c\equiv 1, d\equiv e \bmod r$; that is, $\frac{1}{r}(a, -a, 1, 0; 0)$;
\item[(B)] $a\equiv 1, b+c\equiv 0, d\equiv e \bmod r$; that is, $\frac{1}{r}(1, b, -b, b+1; b+1)$.
\end{enumerate}

If $q=1$,
\begin{enumerate}
\item[(C)]$\frac{1}{r}(a, 1, -a, a+1; a+1)$;
\item[(D)]$\frac{1}{r}(a, -a-1, -a, a+1; -1)$.

\end{enumerate}
This list can be easily derived from Theorem \ref{terminal lemma} and for the proof we refer to \cite[(7.7)]{YPG}. In each case, we may always assume that $0<a<r$ or $0<b<r$ accordingly.

We will discuss case by case.

\bigskip

\noindent {\bf Case (A)}: This gives an isolated hyperquotient singularity of type $\frac{1}{r}(a, -a, 1, 0; 0)$ and $f=xy+g(z^r, t)$
(note that $q=r$ in this case), where $g\in \mm^2$ and $a, r$ are coprime. But such a singularity is terminal by \cite[Theorem 6.5]{KSB}, so this case can be excluded.

\bigskip

\noindent {\bf Case (C)}: Since $a$, $a+1$ are coprime to $r$, we can take an integer $1<k_1<r$ such that $\overline{k_1(a+1)}=1$. Consider 
$$\alpha_{k_1}=\frac{1}{r}(r-k_1+1, k_1, k_1-1, 1).$$
 Then $\alpha_{k_1}(zt)=\frac{k_1}{r}<1$.
There are 3 cases: (C.1) $\alpha_{k_1}\neq \beta, \beta'$; (C.2) $\alpha_{k_1}= \beta$; (C.3) $\alpha_{k_1}= \beta'$.

\medskip

\noindent {\bf Case (C.1)}: If $\alpha_{k_1}\neq \beta, \beta'$, then by Proposition \ref{prop xy}(1), $\alpha_{k_1}(f)=\alpha_{k_1}(xy)-1=\frac{1}{r}$. So there is a monomial $\xx^{\bf m}\in g\in (z,t)^2$ with $\alpha_{k_1}(\xx^{\bf m})=\frac{1}{r}$, but this is absurd.

\medskip

\noindent {\bf Case (C.2)}: If $\alpha_{k_1}= \beta$, then by Proposition \ref{prop xy}(3), $k_1=k_0$ and $\frac{3}{4}<1-\delta<\frac{k_1}{r}<1$.
Then $2k_1> \frac{3}{2}r> r+3$. Note that $r-k_1< 2k_1-r< k_1$. Consider 
$$\alpha_{2k_1-r}=\frac{1}{r}(2r-2k_1+2, 2k_1-r, 2k_1-r-2, 2).$$
 Then $\alpha_{2k_1-r}(zt)=\frac{2k_1-r}{r}<1$ and $\alpha_{2k_1-r}\neq \beta, \beta'$. 
By Proposition \ref{prop xy}(1), $\alpha_{2k_1-r}(f)=\alpha_{2k_1-r}(xy)-1=\frac{2}{r}$. So there is a monomial $\xx^{\bf m}\in g\in (z,t)^2$ with $\alpha_{2k_1-r}(\xx^{\bf m})=\frac{2}{r}$, but this is absurd since by definition $\alpha_{2k_1-r}(z)\geq \alpha_{2k_1-r}(t)=\frac{2}{r}$. Therefore, this case can be excluded.

\medskip

\noindent {\bf Case (C.3)}: If $\alpha_{k_1}= \beta'$, then $\alpha_{r-k_1}=\alpha'_{k_1}= \beta$, which implies that $k_1=r-k_0$ and $1-\delta<\frac{r-k_1}{r}<1$ by Proposition \ref{prop xy}(3). Hence $\frac{k_1}{r}<\delta\leq \frac{1}{3}$ and $k_1<2k_1<r-k_1$. Consider 
$$\alpha_{2k_1}=\frac{1}{r}(r-2k_1+2, 2k_1, 2k_1-2, 2).$$
 In particular, $\alpha_{2k_1}(zt)=\frac{2k_1}{r}<1$ and $\alpha_{2k_1}\neq \beta, \beta'$. So by Proposition \ref{prop xy}(1), $\alpha_{2k_1}(f)=\alpha_{2k_1}(xy)-1=\frac{2}{r}$. So there is a monomial $\xx^{\bf m}\in g\in (z,t)^2$ with $\alpha_{2k_1}(\xx^{\bf m})=\frac{2}{r}$, but this is absurd since by definition  $\alpha_{2k_1}(z)\geq \alpha_{2k_1}(t)=\frac{2}{r}$. Therefore, this case can be excluded.

\bigskip

\noindent {\bf Case (D)}: We can take the integer $k_1=r-1$. Then $\alpha_{k_1}{(zt)}=\frac{k_1}{r}<1$. There are 3 cases: (D.1) $\alpha_{k_1}\neq \beta, \beta'$; (D.2) $\alpha_{k_1}= \beta'$; (D.3) $\alpha_{k_1}= \beta$.

\medskip

\noindent {\bf Case (D.1)}: If $\alpha_{k_1}\neq \beta, \beta'$, then by Proposition \ref{prop xy}(1), $\alpha_{k_1}(f)=\alpha_{k_1}(xy)-1=\frac{1}{r}$. So there is a monomial $\xx^{\bf m}\in g\in (z,t)^2$ with $\alpha_{k_1}(\xx^{\bf m})=\frac{1}{r}$, but this is absurd. 

\medskip

\noindent {\bf Case (D.2)}: If $\alpha_{k_1}= \beta'$, then $\alpha_{r-k_1}=\alpha'_{k_1}= \beta$, which implies that $k_0=r-k_1=1$ and $1-\delta<\frac{1}{r}<1$ by Proposition \ref{prop xy}(3), but this is absurd as $\delta<\frac{1}{2}$.

\medskip

\noindent {\bf Case (D.3)}: If $\alpha_{k_1}= \beta$, then by Proposition \ref{prop xy}(3), $k_1=k_0=r-1$. Now consider $\alpha_{r-2}$, then it is easy to see that $\alpha_{r-2}\neq \beta, \beta'$ as $r>3$. Note that $\alpha_{r-2}{(zt)}\equiv\frac{r-2}{r}\bmod \ZZ$. 

If $\alpha_{r-2}{(zt)}<1$, then $\alpha_{r-2}{(zt)}=\frac{r-2}{r}$ and by Proposition \ref{prop xy}(1), $\alpha_{r-2}(f)=\alpha_{r-2}(xy)-1=\frac{2}{r}$. So there is a monomial $\xx^{\bf m}\in g\in (z,t)^2$ with $\alpha_{r-2}(\xx^{\bf m})=\frac{2}{r}$. As $r>4$, either $z^2\in g$, $\alpha_{r-2}(z)=\frac{1}{r}$, $\alpha_{r-2}(t)=\frac{r-3}{r}$; or $t^2\in g$, $\alpha_{r-2}(z)=\frac{r-3}{r}$, $\alpha_{r-2}(t)=\frac{1}{r}$. We only deal with the former case, the latter one can be reduced to the former one by symmetry by interchanging $x$ with $y$, $z$ with $t$, and $a$ with $-a-1$. The former case implies that $\overline{2a}\equiv 1 \bmod r$, which means that $2a=r+1$. 
The type becomes 
$$\frac{1}{r}\left(\frac{r+1}{2},\frac{r-3}{2},\frac{r-1}{2},\frac{r+3}{2};-1\right).$$
Now consider 
$$\alpha_{r-3}=\frac{1}{r}\left(\frac{r-3}{2},\frac{r+9}{2},\frac{r+3}{2},\frac{r-9}{2}\right)$$
 as $r>9$. Recall that $k_0=r-1$, hence 
$\alpha_{r-3}\neq \beta, \beta'$. Since $\alpha_{r-3}{(zt)}=\frac{r-3}{r}$, by Proposition \ref{prop xy}(1), $\alpha_{r-3}(f)=\alpha_{r-3}(xy)-1=\frac{3}{r}$. So there is a monomial $\xx^{\bf m}\in g\in (z,t)^2$ with $\alpha_{r-3}(\xx^{\bf m})=\frac{3}{r}$. But this is absurd since $\alpha_{r-3}{(z)}>\alpha_{r-3}{(t)}=\frac{r-9}{2r}>\frac{3}{2r}$ as $r>12.$

If $\alpha_{r-2}{(zt)}>1$, then $\alpha_{r-2}{(zt)}=\frac{2r-2}{r}$ and $\overline{2a}+\overline{-2a-2}=2r-2$. Since $a, a+1$ are coprime to $r$, the only solution to this equation is $a\equiv \frac{r-1}{2} \bmod r$. The type becomes $$\frac{1}{r}\left(\frac{r-1}{2},\frac{r-1}{2},\frac{r+1}{2},\frac{r+1}{2};-1\right).$$ Now consider $$\alpha_{r-3}=\frac{1}{r}\left(\frac{r+3}{2},\frac{r+3}{2},\frac{r-3}{2},\frac{r-3}{2}\right).$$ Recall that $k_0=r-1$, hence 
$\alpha_{r-3}\neq \beta, \beta'$. Since $\alpha_{r-3}{(zt)}=\frac{r-3}{r}$, by Proposition \ref{prop xy}(1), $\alpha_{r-3}(f)=\alpha_{r-3}(xy)-1=\frac{3}{r}$. So there is a monomial $\xx^{\bf m}\in g\in (z,t)^2$ with $\alpha_{r-3}(\xx^{\bf m})=\frac{3}{r}$. But this is absurd since $\alpha_{r-3}{(z)}=\alpha_{r-3}{(t)}=\frac{r-3}{2r}>\frac{3}{2r}$ as $r>6.$

\bigskip

\noindent {\bf Case (B)}: Consider $\frac{1}{r}(1,b,-b,b+1; b+1)$ with $b$ coprime to $r$ and $q=\gcd(b+1, r)>0$. If $b+1\equiv 0 \bmod r$, then we get type $\frac{1}{r}(1,-1,1,0; 0)$ which is in case (A), and we already excluded this case.
Hence from now on we assume $1\leq b<b+1<r$.
Consider $$\alpha_{r-1}=\frac{1}{r}(r-1, r-b, b, r-b-1). $$ Note that $\alpha_{r-1}(zt)=\frac{r-1}{r}<1$.
We consider 3 cases: (B.1) $\alpha_{r-1}=\beta'$; (B.2) $\alpha_{r-1}\neq \beta, \beta'$; (B.3) $\alpha_{r-1}= \beta$.

\medskip

\noindent {\bf Case (B.1)}: If $\alpha_{r-1}=\beta'$, then $\alpha_{1}=\beta$ and $k_0=1$, which contradicts $\frac{k_0}{r}>1-\delta$.

\medskip

\noindent {\bf Case (B.2)}: If $\alpha_{r-1}\neq \beta, \beta'$, then by Proposition \ref{prop xy}(1), $\alpha_{r-1}(f)=\alpha_{r-1}(xy)-1=\frac{r-b-1}{r}$. Hence there is a monomial $\xx^{\bf m}\in g(z^q, t)\in (z,t)^2$ with $\alpha_{r-1}(\xx^{\bf m})=\frac{r-b-1}{r}$. Obviously no multiple of $t$ will work, so this monomial has to be $z^n$ with $nb=r-b-1$ for some $n\geq 2$. In particular, $r\geq 3b+1\geq 2b+2$.
We should further consider  $$\alpha_{r-2}=\frac{1}{r}(r-2, r-2b, 2b, r-2b-2)$$ with $\alpha_{r-2}(zt)=\frac{r-2}{r}<1$. Again, there are 3 cases: (B.2.1) $\alpha_{r-2}=\beta'$; (B.2.2) $\alpha_{r-2}\neq \beta, \beta'$; (B.2.3) $\alpha_{r-2}= \beta$.

\medskip

\noindent {\bf Case (B.2.1)}: If $\alpha_{r-2}=\beta'$, then $\alpha_{2}=\beta$ and $k_0=2$, which contradicts $\frac{k_0}{r}>1-\delta$.

\medskip

\noindent {\bf Case (B.2.2)}: If $\alpha_{r-2}\neq \beta, \beta'$, then by 
Proposition \ref{prop xy}(1), $\alpha_{r-2}(f)=\alpha_{r-2}(xy)-1=\frac{r-2b-2}{r}$. Hence there is a monomial $\xx^{\bf m}\in g(z^q, t)\in (z,t)^2$ with $\alpha_{r-2}(\xx^{\bf m})=\frac{r-2b-2}{r}$. This time $z^n$ can not work (because $2nb>r-2b-2$), nor can any multiple of $zt$, so this monomial has to be $t^{n'}$ with $n'(r-2b-2)=r-2b-2$ for some $n'\geq 2$. This implies that $r=2b+2$. Combining with $r\geq 3b+1$, this implies that $r=4$, a contradiction.

\medskip

\noindent {\bf Case (B.2.3)}: If $\alpha_{r-2}= \beta$, then we should further consider $\alpha_{r-3}$. Recall that $r\geq 3b+1$. Note that $r\neq 3b+2$ as $\gcd(r, b+1)>1$. Hence there are two cases:
$$\alpha_{r-3}=\frac{1}{r}(r-3, r-3b, 3b, r-3b-3)$$ if $r\geq 3b+3$, or $$\alpha_{r-3}=\frac{1}{r}(r-3, r-3b, 3b, 2r-3b-3)$$ if $r= 3b+1$.

In the first case, we have $\alpha_{r-3}(zt)=\frac{r-3}{r}<1$ and $\alpha_{r-3}\neq \beta, \beta'$ as $r>5$. Then by 
Proposition \ref{prop xy}(1), $\alpha_{r-3}(f)=\alpha_{r-3}(xy)-1=\frac{r-3b-3}{r}$ and there is a monomial $\xx^{\bf m}\in g(z^q, t)\in (z,t)^2$ with $\alpha_{r-3}(\xx^{\bf m})=\frac{r-3b-3}{r}$. 
This time $z^n$ can not work, nor can any multiple of $zt$, so this monomial has to be $t^{n'}$ with $n'(r-3b-3)=r-3b-3$ for some $n'\geq 2$. This implies that $r=3b+3$. Combining with $nb=r-b-1$ for some $n\geq 2$, this implies that $n\geq 3$ and $r\leq 9$, a contradiction.

In the second case, we consider further $$\alpha_{r-5}=\frac{1}{r}(r-5, 2r-5b, 5b-r, 2r-5b-5)$$ (this holds since $b\geq 3$). We have $\alpha_{r-5}(zt)=\frac{r-5}{r}<1$ and $\alpha_{r-5}\neq \beta, \beta'$ as $r>7$. Then by 
Proposition \ref{prop xy}(1), $\alpha_{r-5}(f)=\alpha_{r-5}(xy)-1=\frac{2r-5b-5}{r}$ and there is a monomial $\xx^{\bf m}\in g(z^q, t)\in (z,t)^2$ with $\alpha_{r-5}(\xx^{\bf m})=\frac{2r-5b-5}{r}$. 
This time $z^n$ can not work, nor can any multiple of $zt$, so this monomial has to be $t^{n'}$ with $n'(2r-5b-5)=2r-5b-5$ for some $n'\geq 2$. This implies that $2r=5b+5$. Combining with $r=3b+1$, this implies that $r=10$, a contradiction. Therefore case (B.2) is excluded.

\medskip

\noindent {\bf Case (B.3)}: If $\alpha_{r-1}= \beta$, then $k_0=r-1$. We consider further $\alpha_{r-2}\neq \beta, \beta'$ as $r>3$. Note that $r\neq 2b, 2b+1$ as $\gcd(r, b)=1$ and $\gcd(r, b+1)>1$. Hence there are 2 cases:
 \[\alpha_{r-2}=\frac{1}{r}(r-2, r-2b, 2b, r-2b-2) \tag{B.3.1}\] if $r\geq 2b+2$; or
\[\alpha_{r-2}=\frac{1}{r}(r-2, 2r-2b, 2b-r, 2r-2b-2)\tag{B.3.2}\] if $r\leq 2b-1$.

\medskip

\noindent {\bf Case (B.3.1)}: In this case, $\alpha_{r-2}(zt)=\frac{r-2}{r}<1$ and by 
Proposition \ref{prop xy}(1), $\alpha_{r-2}(f)=\alpha_{r-2}(xy)-1=\frac{r-2b-2}{r}$. Hence there is a monomial $\xx^{\bf m}\in g(z^q, t)\in (z,t)^2$ with $\alpha_{r-2}(\xx^{\bf m})=\frac{r-2b-2}{r}$. Arguing as before, either $r=2b+2$, or this monomial is $z^n$ with $2nb=r-2b-2$ and $n\geq 2$.

In the first case, we further consider $$\alpha_{r-3}=\frac{1}{r}(r-3, 2r-3b, 3b-r, 2r-3b-3)$$
 (this holds since $b\geq 2$). Then $\alpha_{r-3}(zt)=\frac{r-3}{r}<1$ and $\alpha_{r-3}\neq \beta, \beta'$ as $r>4$. Then by 
Proposition \ref{prop xy}(1), $\alpha_{r-3}(f)=\alpha_{r-3}(xy)-1=\frac{2r-3b-3}{r}$ and there is a monomial $\xx^{\bf m}\in g(z^q, t)\in (z,t)^2$ with $\alpha_{r-3}(\xx^{\bf m})=\frac{2r-3b-3}{r}$. As $2r-3b-3>0$, this monomial has to be $z^{n'}$ with $n'(3b-r)=2r-3b-3$ for some $n'\geq 2$. Combing with $r=2b+2$, this implies that $b\leq 5$ and $r\leq 12$, a contradiction.

In the second case, $r\geq 6b+2$. We further consider $$\alpha_{r-3}=\frac{1}{r}(r-3, r-3b, 3b, r-3b-3).$$ Then $\alpha_{r-3}(zt)=\frac{r-3}{r}<1$ and $\alpha_{r-3}\neq \beta, \beta'$ as $r>4$. Then by 
Proposition \ref{prop xy}(1), $\alpha_{r-3}(f)=\alpha_{r-3}(xy)-1=\frac{r-3b-3}{r}$ and there is a monomial $\xx^{\bf m}\in g(z^q, t)\in (z,t)^2$ with $\alpha_{r-3}(\xx^{\bf m})=\frac{r-3b-3}{r}$. 
But $z^n$ will not work, nor any multiple of $t$ as $r\geq 6b+2$, a contradiction.

\medskip

\noindent {\bf Case (B.3.2)}: In this case, $\alpha_{r-2}(zt)=\frac{r-2}{r}<1$ and by 
Proposition \ref{prop xy}(1), $\alpha_{r-2}(f)=\alpha_{r-2}(xy)-1=\frac{2r-2b-2}{r}$. Hence there is a monomial $\xx^{\bf m}\in g(z^q, t)\in (z,t)^2$ with $\alpha_{r-2}(\xx^{\bf m})=\frac{2r-2b-2}{r}$. Since $r>b+1$, this monomial is $z^n$ with $n(2b-r)=2r-2b-2$ and $n\geq 2$. This implies that $2r\geq 3b+1$. Note that $2r\neq 3b+2$ as $\gcd(r, b+1)>1$. Hence there are two cases for $\alpha_{r-3}$: $$\alpha_{r-3}=\frac{1}{r}(r-3, 2r-3b, 3b-r, 2r-3b-3)$$ if $2r\geq 3b+3$; or 
$$\alpha_{r-3}=\frac{1}{r}(r-3, 2r-3b, 3b-r, 3r-3b-3)$$ if $2r= 3b+1$.

In the first case, $\alpha_{r-3}(zt)=\frac{r-3}{r}<1$ and $\alpha_{r-3}\neq \beta, \beta'$ as $r>4$. Then by 
Proposition \ref{prop xy}(1), $\alpha_{r-3}(f)=\alpha_{r-3}(xy)-1=\frac{2r-3b-3}{r}$ and there is a monomial $\xx^{\bf m}\in g(z^q, t)\in (z,t)^2$ with $\alpha_{r-3}(\xx^{\bf m})=\frac{2r-3b-3}{r}$. But $z^n$ or any multiple of $zt$ can not work, hence this monomial is $t^{n'}$ for some $n'\geq 2$ which implies that $2r=3b+3$. Combining with $n(2b-r)=2r-2b-2$ for $n\geq 2$, this implies that $n\geq 3$. If $n\geq 4$, then it is easy to show that $r\leq 12$ by this two equations, a contradiction. If $n=3$, then $(r, b+1)=(18, 12)$. But recall that $q=\gcd(b+1, r)$ divides $n$ by construction as $z^n\in g(z^q, t)$, this is also absurd.

In the second case, further consider $$\alpha_{r-4}=\frac{1}{r}(r-4, 3r-4b, 4b-2r, 3r-4b-4)$$ (since $b\geq 5$). Since $\alpha_{r-4}(zt)=\frac{r-4}{r}<1$ and $\alpha_{r-4}\neq \beta, \beta'$ as $r>5$, then by 
Proposition \ref{prop xy}(1), $\alpha_{r-4}(f)=\alpha_{r-4}(xy)-1=\frac{3r-4b-4}{r}$ and there is a monomial $\xx^{\bf m}\in g(z^q, t)\in (z,t)^2$ with $\alpha_{r-4}(\xx^{\bf m})=\frac{3r-4b-4}{r}$. 
If  $3r-4b-4\neq 0$, then  any multiple of $t^2$ or $zt$ can not work, hence this monomial is $z^{n'}$ for some $n'\geq 2$ which implies that
$3r-4b-4\geq 2(4b-2r)$. In conclusion,  $3r-4b-4=0$ or $3r-4b-4\geq 2(4b-2r)$ holds. Combining with $2r=3b+1$, it is easy to see that $r\leq 8$, a contradiction.

\bigskip

Therefore, the $cA$ case is excluded.
 
\subsection{The odd case}\label{odd}
 
In this subsection, we consider the odd case in Proposition \ref{f cases}: $f=x^2+y^2+g(z,t)$ with $g\in \mm^3$ and $a\neq b$.

By Proposition \ref{prop x2}(2)(4), we can list all possible types by Theorem \ref{terminal lemma}, and one of the followings holds:
$\frac{1}{2}(0,1,1,1;0)$; or $\frac{1}{r}(1, \frac{r+2}{2},\frac{r-2}{2}, 2;2)$ with $4|r.$
For the proof we refer to \cite[(7.10)]{YPG}.

Here we only need to exclude the second case. Recall that $r>12$. Consider
$\alpha_{r-2}=\frac{1}{r}(r-2,r-2,2,r-4)$. We need to consider $2$ cases: 
 (1) $\alpha_{r-2}\neq \beta, \beta'$; (2) $\alpha_{r-2}= \beta$ or $\beta'$.



\medskip

\noindent {\bf Case (1)}: If $\alpha_{r-2}\neq \beta, \beta'$, then since $\alpha_{r-2}(yzt)=\frac{2r-4}{r}<\alpha_{r-2}(x)+1=\frac{2r-2}{r}$, by 
Proposition \ref{prop x2}(1), $\alpha_{r-2}(f)=2\alpha_{r-2}(x)-1=\frac{r-4}{r}$ and there is a monomial $\xx^{\bf m}\in g\in (z,t)^2$ with $\alpha_{r-2}(\xx^{\bf m})=\frac{r-4}{r}$. Note that the only possible monomial with weight $\frac{r-4}{r}$ is $z^{\frac{r-4}{2}}$, but it is not in the same eigenspace as $f$ since $\frac{r-4}{2}c=\frac{r-4}{2}\cdot\frac{r-2}{2} \equiv 2-\frac{r}{2}\not \equiv 2\equiv e \bmod r$.

\medskip

\noindent {\bf Case (2)}: If $\alpha_{r-2}= \beta$ or $\beta'$, then we further
consider $\alpha_{r-4}=\frac{1}{r}(r-4,r-4,4,r-8)$. Note that $\alpha_{r-4}\neq \beta, \beta'$ as $r>6$, and
$\alpha_{r-4}(yzt)=\frac{2r-8}{r}<\alpha_{r-4}(x)+1=\frac{2r-4}{r}$. Then by 
Proposition \ref{prop x2}(1), $\alpha_{r-4}(f)=2\alpha_{r-4}(x)-1=\frac{r-8}{r}$ and there is a monomial $\xx^{\bf m}\in g\in (z,t)^2$ with $\alpha_{r-4}(\xx^{\bf m})=\frac{r-8}{r}$. Note that the only possible monomial with weight $\frac{r-8}{r}$ is $z^{\frac{r-8}{4}}$, but it is not in the same eigenspace as $f$ since $\frac{r-8}{4}c\not \equiv 2\equiv e \bmod r$. This can be seen by $\frac{r-8}{4}c=\frac{r-8}{4}\cdot \frac{r-2}{2}= 2-r+\frac{r^2-2r}{8}$ where  $\frac{ r-2}{8}$ is not an integer.

\medskip

Therefore, the odd case is excluded.

\subsection{The $cD$-$E$ case}\label{cDE}
 
In this subsection, we consider the remaining cases in Proposition \ref{f cases}: $f=x^2+g(y, z,t)$ with $g\in \mm^3$.

By Proposition \ref{prop x2}(2)(4), we can list all possible types by Theorem \ref{terminal lemma}, and one of the followings holds (after possibly interchanging $y, z,t$):

If $a\equiv e\equiv 0 \bmod r$ and $b,c,d$ are coprime to $r$, then
 \begin{enumerate}
 \item[(a)] $\frac{1}{r}(0, b,-b,1;0)$ with $b$ coprime to $r$.
 \end{enumerate}

If $q=2=\gcd(d,r)=\gcd(e,r)$ and $a, b,c$ are coprime to $r$, then
 \begin{enumerate}
 \item[(b)] $\frac{1}{r}(a, -a,1, 2a;2a)$ with $r$ even and $a$ coprime to $r$;
 \item[(c)] $\frac{1}{r}(1, b,-b, 2;2)$ with $r$ even and $b$ coprime to $r$.
 \end{enumerate}

If $r$ is odd and $a,b,c,d,e$ are coprime to $r$, then
 \begin{enumerate}
 \item[(d)] $\frac{1}{r}(\frac{r-1}{2}, \frac{r+1}{2},c,-c;-1)$ with $r$ odd and $c$ coprime to $r$;
 \item[(e)] $\frac{1}{r}(a,-a,2a,1;2a)$ with $r$ odd and $a$ coprime to $r$;
 \item[(f)] $\frac{1}{r}(1,b,-b,2;2)$ with $r$ odd and $b$ coprime to $r$.
 \end{enumerate}
In each case, we may always assume that $0<a,b,c<r$ accordingly.

We will discuss case by case.

\bigskip

\noindent {\bf Case (b)}: Take $0<k_1<r$ such that $\overline{k_1a}=\frac{r+2}{2}<r$, then $\alpha_{k_1}=\frac{1}{r}(\frac{r+2}{2},\frac{r-2}{2}, k_1, 2 )$. We need to consider $2$ cases: (b.1) $\alpha_{k_1}\neq \beta, \beta'$; (b.2) $\alpha_{k_1}= \beta$ or $\beta'$.

\medskip

\noindent {\bf Case (b.1)}: If $\alpha_{k_1}\neq \beta, \beta'$, since 
$\alpha_{k_1}(yzt)<\alpha_{k_1}(x)+1$, by Proposition \ref{prop x2}(1), $\alpha_{k_1}(f)=2\alpha_{k_1}(x)-1=\frac{2}{r}$, but no monomial in $\mm^3$ has weight $\leq \frac{2}{r}$, a contradiction.

\medskip

\noindent {\bf Case (b.2)}: If $\alpha_{k_1}= \beta$ or $\beta'$, then we consider further $0<k_2<r$ such that $\overline{k_2a}=\frac{r+4}{2}<r$, then $\alpha_{k_2}=\frac{1}{r}(\frac{r+4}{2},\frac{r-4}{2}, k_2, 4 )$. 
Note that $\alpha_{k_2}\neq \beta, \beta'$, otherwise $k_1+k_2\equiv 0\bmod r$, which implies that $   \frac{r+2}{2}+ \frac{r+4}{2}\equiv 0\bmod r$, a contradiction.
Since 
$\alpha_{k_2}(yzt)<\alpha_{k_2}(x)+1$, by Proposition \ref{prop x2}(1), $\alpha_{k_2}(f)=2\alpha_{k_2}(x)-1=\frac{4}{r}$, but the only possible monomial in $\mm^3$ has weight $\leq \frac{4}{r}$ is $z^4$ with $k_2=1$ (recall that $r>12$). In this case, $a\equiv \frac{r+4}{2}\bmod r$ and hence $ k_1\frac{r+4}{2}\equiv \frac{r+2}{2}\bmod r$. This implies that $r|4k_1-2$. $r>12$ implies that $k_1\geq 4$. Note that $r> k_1$ and $r$ is even, so either $r=4k_1-2$ or $3r=4k_1-2$, in particular, in both case, $3r\geq  4k_1-2\geq \frac{7}{2}k_1$. Recall that $\alpha_{k_1}= \beta$ or $\beta'$ implies that $k_1=k_0$ or $r-k_0$. If $k_1=r-k_0$, then $\frac{k_0}{r}>1-\delta$ implies that $r>4k_1$, a contradiction. If $k_1=k_0$, then $\frac{k_0}{r}\leq \frac{6}{7}$ by the above calculation, again a contradiction.

\bigskip

\noindent {\bf Case (c)}: This case is similar to case (b). Take $k_1=\frac{r+2}{2}$ and consider $\alpha_{k_1}=\frac{1}{r}(\frac{r+2}{2},\overline{k_1b}, r-\overline{k_1b}, 2 )$. We need to consider $2$ cases: (c.1) $\alpha_{k_1}\neq \beta, \beta'$; (c.2) $\alpha_{k_1}= \beta$ or $\beta'$.


\medskip

\noindent {\bf Case (c.1)}: In this case the argument is the same as case (b.1). If $\alpha_{k_1}\neq \beta, \beta'$, since 
$\alpha_{k_1}(yzt)<\alpha_{k_1}(x)+1$, by Proposition \ref{prop x2}(1), $\alpha_{k_1}(f)=2\alpha_{k_1}(x)-1=\frac{2}{r}$, but no monomial in $\mm^3$ has weight $\leq \frac{2}{r}$, a contradiction.

\medskip

\noindent {\bf Case (c.2)}: If $\alpha_{k_1}= \beta$ or $\beta'$, then $k_1=k_0$ or $r-k_0$. Recall that $k_1=\frac{r+2}{2}$, so  $\frac{k_0}{r}\leq \frac{r+2}{2r}\leq \frac{5}{6}$ as $r>2$, and get a contradiction by $\frac{k_0}{r}>1-\delta$.


\bigskip

\noindent {\bf Case (d)}: Take $k_1=r-1$, then $\alpha_{k_1}=\frac{1}{r}(\frac{r+1}{2},\frac{r-1}{2}, \overline{k_1c}, r-\overline{k_1c} )$. We need to consider $2$ cases: (d.1) $\alpha_{k_1}\neq \beta, \beta'$; (d.2) $\alpha_{k_1}= \beta$ or $\beta'$.

\medskip

\noindent {\bf Case (d.1)}: If $\alpha_{k_1}\neq \beta, \beta'$, since 
$\alpha_{k_1}(yzt)<\alpha_{k_1}(x)+1$, by Proposition \ref{prop x2}(1), $\alpha_{k_1}(f)=2\alpha_{k_1}(x)-1=\frac{1}{r}$, but no monomial in $\mm^3$ has weight $\leq \frac{1}{r}$, a contradiction.

\medskip

\noindent {\bf Case (d.2)}: If $\alpha_{k_1}= \beta$ or $\beta'$, 
then we consider $0<k_2, k_3<r$ such that $\overline{k_2a}=\frac{r+3}{2}<r$ 
and $\overline{k_3a}=\frac{r+5}{2}<r$. 
Then 
$\alpha_{k_2}=\frac{1}{r}(\frac{r+3}{2},\frac{r-3}{2}, \overline{k_2c}, r-\overline{k_2c} )$ and
$\alpha_{k_3}=\frac{1}{r}(\frac{r+5}{2},\frac{r-5}{2}, \overline{k_3c}, r-\overline{k_3c} )$.
Note that $\alpha_{k_2}\neq \beta$ or $\beta'$, otherwise $k_1+k_2\equiv 0$ and $\overline{ak_1+ak_2}=2\equiv 0 \bmod r$, which is absurd. Since 
$\alpha_{k_2}(yzt)<\alpha_{k_2}(x)+1$, by Proposition \ref{prop x2}(1), $\alpha_{k_2}(f)=2\alpha_{k_2}(x)-1=\frac{3}{r}$, but the only possible monomial in $\mm^3$ has weight $\frac{3}{r}$ is $z^3$ with $\overline{k_2c}=1$ (after possibly interchanging $z,t$).
Similarly, $\alpha_{k_3}\neq \beta$ or $\beta'$, otherwise $k_1+k_3\equiv 0$ and $\overline{ak_1+ak_3}=3\equiv 0 \bmod r$, which is absurd. Since 
$\alpha_{k_3}(yzt)<\alpha_{k_3}(x)+1$, by Proposition \ref{prop x2}(1), $\alpha_{k_3}(f)=2\alpha_{k_3}(x)-1=\frac{5}{r}$. In order to have a monomial in $\mm^3$ with weight $\frac{5}{r}$, one of  $\overline{k_3c}$ and $r-\overline{k_3c}$ is 1. Therefore $k_2\pm k_3\equiv 0 \bmod r$ and this implies that $\frac{r+3}{2}\pm \frac{r+5}{2}\equiv 0 \bmod r$, a contradiction.

\bigskip

\noindent {\bf Case (e)}: Take $0<k_1<r$ such that $\overline{k_1a}=\frac{r+1}{2}$, then $\alpha_{k_1}=\frac{1}{r}(\frac{r+1}{2},\frac{r-1}{2}, 1, k_1)$. We need to consider $2$ cases: (e.1) $\alpha_{k_1}\neq \beta, \beta'$; (e.2) $\alpha_{k_1}= \beta$ or $\beta'$.

\medskip

\noindent {\bf Case (e.1)}:  In this case the argument is the same as case (d.1). If $\alpha_{k_1}\neq \beta, \beta'$, since 
$\alpha_{k_1}(yzt)<\alpha_{k_1}(x)+1$, by Proposition \ref{prop x2}(1), $\alpha_{k_1}(f)=2\alpha_{k_1}(x)-1=\frac{1}{r}$, but no monomial in $\mm^3$ has weight $\leq \frac{1}{r}$, a contradiction.

\medskip

\noindent {\bf Case (e.2)}: If $\alpha_{k_1}= \beta$ or $\beta'$, 
then we consider $0<k_2, k_3<r$ such that $\overline{k_2a}=\frac{r+3}{2}<r$ 
and $\overline{k_3a}=\frac{r+5}{2}<r$. 
Then 
$\alpha_{k_2}=\frac{1}{r}(\frac{r+3}{2},\frac{r-3}{2},3, k_2)$ and
$\alpha_{k_3}=\frac{1}{r}(\frac{r+5}{2},\frac{r-5}{2}, 5, k_3 )$.
We can get a contradiction similarly as case (d.2).
Note that $\alpha_{k_2}\neq \beta$ or $\beta'$, otherwise $k_1+k_2\equiv 0$ and $\overline{ak_1+ak_2}=2\equiv 0 \bmod r$, which is absurd. Since 
$\alpha_{k_2}(yzt)<\alpha_{k_2}(x)+1$, by Proposition \ref{prop x2}(1), $\alpha_{k_2}(f)=2\alpha_{k_2}(x)-1=\frac{3}{r}$, but the only possible monomial in $\mm^3$ has weight $\frac{3}{r}$ is $t^3$ with $k_2=1$.
Similarly, $\alpha_{k_3}\neq \beta$ or $\beta'$, otherwise $k_1+k_3\equiv 0$ and $\overline{ak_1+ak_3}=3\equiv 0 \bmod r$, which is absurd.  Since 
$\alpha_{k_3}(yzt)<\alpha_{k_3}(x)+1$, by Proposition \ref{prop x2}(1), $\alpha_{k_3}(f)=2\alpha_{k_3}(x)-1=\frac{5}{r}$, but the only possible monomial in $\mm^3$ has weight $\frac{5}{r}$ is $t^5$ with $k_3=1$. This is absurd as $k_2\neq k_3$.

\bigskip

\noindent {\bf Case (f)}: Take $k_1=\frac{r+1}{2}$, then $\alpha_{k_1}=\frac{1}{r}(\frac{r+1}{2},\overline{k_1b}, r-\overline{k_1b}, 1)$. We need to consider $2$ cases: (f.1) $\alpha_{k_1}\neq \beta, \beta'$; (f.2) $\alpha_{k_1}= \beta$ or $\beta'$.

\medskip

\noindent {\bf Case (f.1)}:  In this case the argument is the same as case (d.1). If $\alpha_{k_1}\neq \beta, \beta'$, since 
$\alpha_{k_1}(yzt)<\alpha_{k_1}(x)+1$, by Proposition \ref{prop x2}(1), $\alpha_{k_1}(f)=2\alpha_{k_1}(x)-1=\frac{1}{r}$, but no monomial in $\mm^3$ has weight $\leq \frac{1}{r}$, a contradiction.

\medskip

\noindent {\bf Case (f.2)}: If $\alpha_{k_1}= \beta$ or $\beta'$, 
then we consider $0<k_2, k_3<r$ such that $\overline{k_2a}=\frac{r+3}{2}<r$ 
and $\overline{k_3a}=\frac{r+5}{2}<r$. 
Then 
$\alpha_{k_2}=\frac{1}{r}(\frac{r+3}{2},\overline{k_2b}, r-\overline{k_2b}, 3)$ and
$\alpha_{k_3}=\frac{1}{r}(\frac{r+5}{2},\overline{k_3b}, r-\overline{k_3b}, 5)$. We can get a contradiction similarly as case (d.2).
Note that $\alpha_{k_2}\neq \beta$ or $\beta'$, otherwise $k_1+k_2\equiv 0$ and $\overline{ak_1+ak_2}=2\equiv 0 \bmod r$, which is absurd. Since 
$\alpha_{k_2}(yzt)<\alpha_{k_2}(x)+1$, by Proposition \ref{prop x2}(1), $\alpha_{k_2}(f)=2\alpha_{k_2}(x)-1=\frac{3}{r}$, but the only possible monomial in $\mm^3$ has weight $\frac{3}{r}$ is $z^3$ with $r-\overline{k_2b}=1$.
Similarly, $\alpha_{k_3}\neq \beta$ or $\beta'$, otherwise $k_1+k_3\equiv 0$ and $\overline{ak_1+ak_3}=3\equiv 0 \bmod r$, which is absurd.  Since 
$\alpha_{k_3}(yzt)<\alpha_{k_3}(x)+1$, by Proposition \ref{prop x2}(1), $\alpha_{k_3}(f)=2\alpha_{k_3}(x)-1=\frac{5}{r}$, but the only possible monomial in $\mm^3$ has weight $\frac{5}{r}$ is $z^5$ with $r-\overline{k_3b}=1$. This is absurd as $k_2\neq k_3$.

\bigskip

\noindent {\bf Case (a)}: Finally we consider case (a), $\frac{1}{r}(0, b,-b,1;0)$ with $b$ coprime to $r$. Note that $\alpha_k=\frac{1}{r}(0, \overline{bk}, r-\overline{bk}, k)$ for $1\leq k\leq r-1$ and $\alpha_k(g)\equiv 2\alpha_k(x)=0 \bmod \ZZ$. So $\alpha_k(g)\in \ZZ_{>0}$. In this case we can get the following condition for $g(y,z,t)$.

\begin{claim}\label{lem g=1} In case (a), for $1\leq k\leq r-1$ such that $\alpha_k\not \equiv \beta\bmod \ZZ^4$, $\alpha_k(g)=1$.
\end{claim}

\begin{proof}
Take $i=\frac{\alpha_k(g)}{2}$ or $\frac{\alpha_k(g)+1}{2}$ respectively if $\alpha_k(g)$ is even or odd. Consider $\gamma=\alpha_k+(i,0,0,0)$. Then $\gamma(f)=\min\{\alpha_k(g), 2i\}=\alpha_k(g)$
and $\gamma(xyzt)=\alpha_k(xyzt)+i=\frac{r+k}{r}+i$.
By the assumption, $\gamma\neq \beta$, hence by Rule I, $\gamma(xyzt)>\gamma(f)+1$, which implies that $\alpha_k(g)<\frac{k}{r}+i<1+i$. Writing out the definition of $i$, it is easy to see that $\alpha_k(g)=1$ is the only solution.
\end{proof}

Now come back to case (a).
Take $0<k_1<r$ such that $\overline{k_1b}=\frac{r-1}{2}$. We may assume that $k_1\geq \frac{r}{2}$ by possibly interchanging $z,t$. Consider $\alpha_{k_1}=\frac{1}{r}(0, \frac{r-1}{2}, \frac{r+1}{2}, k_1)$.

If $\alpha_{k_1}\not \equiv \beta\bmod \ZZ^4$, then by Claim \ref{lem g=1}, $\alpha_{k_1}(g)=1$, which means that there is a monomial in $(y,z,t)^3$ with weight $1$. But all monomials in $(y,z,t)^3$ have weights $\geq \frac{3(r-1)}{2r}>1$, a contradiction.

If $\alpha_{k_1} \equiv \beta\bmod \ZZ^4$, then take $0<k_2<r$ such that $\overline{k_2b}=\frac{r-3}{2}$. Consider $\alpha_{k_2}=\frac{1}{r}(0, \frac{r-3}{2}, \frac{r+3}{2}, k_2)$ and $\alpha_{r-k_2}=\frac{1}{r}(0, \frac{r+3}{2}, \frac{r-3}{2}, r-k_2)$. It is easy to see that $\alpha_{k_2}, \alpha_{r-k_2}\not\equiv \beta\bmod \ZZ^4$ as $\frac{r-3}{2}\pm \frac{r-1}{2}\not \equiv 0 \bmod r$.
 Hence by Claim \ref{lem g=1}, $\alpha_{k_2}(g)=\alpha_{r-k_2}(g)=1$. If $k_2\geq \frac{r}{2}$, then all monomials in $(y,z,t)^3$ have $\alpha_{k_2}$-weights $\geq \frac{3(r-3)}{2r}>1$; if $k_2< \frac{r}{2}$, then all monomials in $(y,z,t)^3$ have $\alpha_{r-k_2}$-weights $\geq \frac{3(r-3)}{2r}>1$ as $r>9$. So this case is excluded.
 
 \bigskip
 
 Therefore, the $cD$-$E$ case is excluded.

\section{The $1$-gap theorem for $3$-dimensional non-canonical singularities: the general case}\label{section general case}

Firstly we prove the $1$-gap theorem for surfaces, which  may be well-known to experts.
\begin{lem}\label{gap dim 2}
Let $S$ be a normal quasi-projective $\bQ$-Gorenstein surface. Assume that $\mld(S)<1$, then $\mld(S)\leq \frac{2}{3}$.
\end{lem}
\begin{proof}
We may assume that $S$ has klt singularities. Take $\pi: S'\to S$ to be the minimal resolution of $S$ and write $K_{S'}+\sum_i a_iC_i=\pi^*K_S$ where $C_i$ are distinct exceptional curves and $1>a_i\geq 0$. 
Since $\mld(S)<1$, it has worse than du Val singularities, and hence there exists an exceptional curve $C$ with $C^2\leq -3$. We may assume that $C_1=C$. Then by the genus formula,
\begin{align*}
-2{}&\leq (K_{S'}+C_1)\cdot C_1\leq (K_{S'}+C_1+\sum_{i\neq1} a_iC_i)\cdot C_1\\
{}&=(1-a_1)C_1^2\leq -3(1-a_1).
\end{align*}
This implies that $a(C; S)=1-a_1\leq \frac{2}{3}$.
\end{proof}

\begin{remark}The number $\frac{2}{3}$ is optimal in  Lemma \ref{gap dim 2}. In fact, the minimal log discrepancy of a cyclic quotient singularity of type $\frac{1}{3}(1,1)$ is $\frac{2}{3}$.
\end{remark}

Now we are ready to prove Theorem \ref{gap main},  the $1$-gap theorem for $3$-dimensional non-canonical singularities.
\begin{proof}[Proof of Theorem \ref{gap main}]

Take $\delta=\delta_0$ as in Theorem \ref{no hyperquotient}, where $\delta_0$ is the constant from Lemma \ref{non-canonical lemma}. Recall that $\delta\leq \delta_3$, where $\delta_3$ is the constant from Corollary \ref{acc quot 3}.

Assume that there is a normal quasi-projective $\bQ$-Gorenstein $3$-fold $X$ with $1-\delta<\mld(X)<1$, in particular, $X$ is klt.
By Theorem \ref{X<Y}, after replacing $X$, we may assume that $X$ is extremely non-canonical. Let $E_0$ be the unique exceptional divisor over $X$ such that $a(E_0;X)<1$. Let $c_{E_0}(X)$ denote the center of $E_0$ on $X$. By definition, $X$ is terminal outside $c_{E_0}(X)$. As $3$-dimensional terminal singularities are isolated, by shrinking $X$, we may assume that $X$ is smooth outside $c_{E_0}(X)$.

If the center $c_{E_0}(X)=C$ is a curve, we can take a general hyperplane section $H\subset X$ intersecting $C$. Here $H$ is a normal quasi-projective  $\bQ$-Gorenstein surface and $\mld (H)\geq \mld (X)$ by the Bertini theorem (\cite[Lemma 5.17]{KM}). On the other hand, by the inversion of adjucntion (\cite[Corollary 1.4.5]{BCHM}), $\mld(H)\leq a(E_0; X, H)=a(E_0; X)=\mld(X)$. Hence $1-\delta<\mld(H)<1$. But this contradicts Lemma \ref{gap dim 2}. 

So we may assume that $c_{E_0}(X)=P$ and $(P\in X)$ is an isolated extremely non-canonical klt singularity with $\mld(X)>1-\delta$. Denote $r$ to be the minimal positive integer such that $rK_X$ is Cartier and take $(Q\in Y)$ to be the canonical index $1$ cover of $(P\in X)$. Then $(Q\in Y)$ is an isolated index one canonical singularity. By the classification of $3$-dimensional index one canonical singularities (see \cite{Reid80} or \cite[5.3]{KM}), there are 3 cases: $(Q\in Y)$ is smooth; $(Q\in Y)$ is an isolated cDV singularity; $(Q\in Y)$ is an isolated non-cDV singularity.

If $(Q\in Y)$ is an isolated non-cDV singularity, then there exists an exceptional divisor $E'$ over $Y$ centered at $Q$ such that $a(E'; Y)=1$. Hence by the ramification formula (see, for example, the calculation in \cite[(20.3) Proposition]{Kollar92} or \cite[Proposition 5.20]{KM}), there exists an exceptional divisor $E$ over $X$ such that $n\cdot a(E; X)=a(E'; Y)=1$ for some positive integer $n$. Since $X$ is extremely non-canonical, $n>1$. Therefore $\mld(X)\leq a(E; X)\leq \frac{1}{2}$, a contradiction.

If $(Q\in Y)$ is smooth, then $(P\in X)$ is an isolated  cyclic quotient singularity and this contradicts Corollary \ref{acc quot 3}.

If $(Q\in Y)$ is an isolated cDV singularity, then we get a contradiction by Theorem \ref{no hyperquotient}.

In summary, such a normal quasi-projective  $\bQ$-Gorenstein $3$-fold $X$ with $1-\delta<\mld(X)<1$ does not exist, and the theorem is proved.
\end{proof}

\section{Boundedness of global indices of klt Calabi--Yau $3$-folds}\label{section bounded}
In this section, we give applications for Theorem \ref{gap main}. We show that the set of all non-canonical klt Calabi--Yau $3$-folds are bounded modulo flops, and the global indices of all klt Calabi--Yau $3$-folds are bounded from above. To be more precise, we show the followings:

\begin{thm}\label{bdd flop cy}
The set of non-canonical klt Calabi--Yau $3$-folds forms a bounded family modulo flops.
\end{thm}

\begin{cor}\label{bdd index cy}
There exists a positive integer $m$ such that for any klt Calabi--Yau $3$-fold $X$, $mK_X\sim 0$.
\end{cor}

Recall that a variety is {\it uniruled} if it is covered by rational curves. The following lemma may be well-known to experts. 
\begin{lem}\label{non-canonical=uniruled}
Let $X$ be a klt Calabi--Yau variety. Then $X$ is non-canonical if and only if $X$ is uniruled.
\end{lem}
\begin{proof}
Suppose that $X$ is uniruled. Then by taking a resolution $\phi: Y \to X$, $Y$ is again uniruled, which implies that $K_Y$ is not pseudo-effective. Assume, to the contrary, that $X$ is canonical, then $K_Y\geq \phi^*K_X$, and therefore $K_X$ is not psuedo-effective, which contradicts $K_X\equiv 0$.

Suppose that $X$ is non-canonical, then there exists an exceptional divisor $E$ over $X$ with log discrepancy $<1$. By \cite[Corollary 1.4.3]{BCHM}, there is a projective birational morphism $\phi: Y \to X$ extracting $E$. We can write $K_Y+aE=\phi^*K_X\equiv 0$ with $a>0$, which means that $K_Y$ is not psuedo-effective. But this implies that $Y$ is uniruled by \cite{BDPP}, and so is $X$. 
\end{proof}

The key is to show the following proposition (comparing with \cite[Corollary 4.2]{rccy3}).

\begin{prop}\label{cor.bdd.lcy3fold}
Fix positive real numbers $\epsilon$, $\delta$. 
Then, the set of log pairs $(X,B)$ satisfying
\begin{enumerate}
\item $(X,B)$ is an $\epsilon$-lc log Calabi--Yau 
pair of dimension $3$,
\item there is a component of $\text{\rm Supp}\,B$ which is uniruled, and
\item the non-zero coefficients of $B$ are at least $\delta$,
\end{enumerate}
forms a log bounded family modulo flops.
\end{prop}
\begin{proof}
We may replace $X$ by its small $\bQ$-factorialization (by \cite[Corollary 1.4.3]{BCHM}) and assume that $X$ is $\bQ$-factorial.
We may write $B=B'+dD$, where $D$ is a uniruled component of $B$ and $d>0$.
We can run a $(K_X+B')$-MMP 
with scaling of an ample divisor
which ends with a Mori fiber space $f:Y \to Z$.
Denote by $B_{Y}$, $D_Y$ the strict transforms of $B$,  $D$ on $Y$.
Since  $K_X+B'+dD \equiv 0$, and we are running a $(K_X+B')$-MMP, 
it follows that $D_{Y}$ is uniruled and dominates $Z$. Also note that $(Y, B_Y)$ is again 
an $\epsilon$-lc log Calabi--Yau pair, 
and coefficients of $B_Y$ are at least $\delta$. 

We claim that $Z$ is in a bounded family. If $Z$ is a point, then there is nothing to prove. If $\dim Z=1$, then according to Ambro's canonical bundle formula (see \cite[Theorem 3.1]{FG}), $-K_Z$ is pseudo-effective, which means that $Z$ is either $\bP^1$ or an elliptic curve, which is in a bounded family. If $\dim Z=2$, then by \cite[Corollary 1.7]{Bir16a}, there exists an effective 
$\bR$-divisor $\Delta$ such that $(Z, \Delta)$ is $\epsilon'$-klt and
$K_Z+\Delta\sim_\bR 0$, where $\epsilon'$ is a positive number depending only on $\epsilon$.
Since $D_{Y}$ dominates $Z$, $Z$ is also uniruled. In particular,  by Lemma \ref{non-canonical=uniruled}, we are not in the case that $K_Z\equiv 0$, $\Delta=0$, and $Z$ has canonical singularities. Therefore, by \cite[Theorem 6.9]{Ale94}, such $Z$ is in a bounded family.

As $Z$ is in a bounded family,  we may find a very ample divisor $A$ on $Z$, and a positive integer $r$ independent of $X$ such that $A^{\dim Z}\leq r$. Here if $Z$ is a point, we just formally define 
$A^{\dim Z}=1$.
Therefore, $(Y, B_Y)\to Z$ is a $(3, r, \epsilon)$-Fano type log Calabi--Yau fibration in the sense of \cite[Definition 1.1]{Birkar18}. Hence by \cite[Theorem 1.2]{Birkar18}, such $Y$ is in a bounded family. 
Here instead of using  \cite[Theorem 1.2]{Birkar18}, we can also use \cite[Theorem 1.4]{BirkarBAB1} and \cite[Theorem 4.6]{rccy3} to conclude the boundedness of $Y$.
As the coefficients of $B_Y$ are at least $\delta$, it is easy to see that the pair
$(Y, B_Y)$ is in a log bounded family. In fact, as $Y$ is bounded, we can find a very ample divisor $H$ on $Y$ such that $H^3\leq r'$ and $H^2\cdot (-K_Y)\leq r'$ for some positive integer $r'$ independent of $Y$, then $H^2\cdot \Supp(B_Y)\leq \frac{1}{\delta}H^2\cdot B_Y= \frac{1}{\delta} H^2\cdot (-K_Y)\leq \frac{r'}{\delta}$, and we can use \cite[Lemma 2.20]{BirkarBAB1} to conclude the log boundedness. For any prime divisor $E$ on $X$ which is exceptional over $Y,$ we have
$$a(E;Y,B_Y ) = a(E;X,B) \leq a(E;X,0) = 1.$$
Hence, $(X, B)$ is in a log bounded family modulo flops by \cite[Proposition 4.8]{rccy3} by extracting all such $E$ simultaneously in the log bounded family of $(Y, B_Y)$.
\end{proof}


Now we are ready to present the proof of Theorem \ref{bdd flop cy}. It is almost the same as that of \cite[Theorem 5.1]{rccy3}, the essential modifications are that we remove the condition on minimal log discrepancies by Theorem \ref{gap main}, and remove the rational connectedness condition by Proposition \ref{cor.bdd.lcy3fold}.

\begin{proof}[Proof of Theorem \ref{bdd flop cy}]
Consider a non-canonical klt Calabi--Yau $3$-fold $X$. By Theorem \ref{gap main}, there exists a constant $0<\delta<1$ independent of $X$ such that $\mld(X)\leq 1-\delta$.
By \cite[Corollary 1.4.3]{BCHM}, we may take a projective birational morphism $\pi: Y \to X$ extracting only one exceptional divisor $E$ with log discrepancy $a=a(E; X)\leq 1-\delta$. We can write
\[
K_Y+(1-a)E = \pi^* K_X\equiv 0.
\]
Also by Global ACC  \cite[Theorem 1.5]{HMX14} (see \cite[Lemma 3.12]{rccy3}), there exists a constant $\epsilon \in (0,\frac{1}{2})$ such that $X$ is $(2\epsilon)$-lc, and therefore $(Y, (1-a)E)$ is a $(2\epsilon)$-lc log Calabi--Yau pair with $1-a\geq \delta>0$. 
Here $E$ is uniruled by \cite{HM07}.

Now we can apply Proposition \ref{cor.bdd.lcy3fold} to see that the pairs $(Y, (1-a)E)$ are log bounded modulo flops.
That is, there are finitely many quasi-projective normal varieties $\mathcal{W}_i$, a reduced divisor $\mathcal{E}_i$ on $\mathcal W_i$, and a projective morphism $\mathcal{W}_i\to S_i$, where $S_i$ is a normal variety of finite type and $\mathcal{E}_i$ does not contain any fiber, such that
for every $(Y, (1-a)E)$, there is an index $i$, a closed point $s \in S_i$, and a small birational
map $f : {\mathcal{W}_{i,s}} \dashrightarrow Y$
such that $\mathcal{E}_{i,s}=f_*^{-1}E$. We may assume that the set of points $s$ corresponding to such $Y$ is dense in each $S_i$.
We may just consider a fixed index $i$ and ignore the index in the following argument.

Now we are going to prove that $X$ is bounded modulo flops by contracting $E$ simultaneously in the bounded family $(\mathcal{W}, \mathcal{E})$. The argument is exactly the same as the latter half of   \cite[Theorem 5.1]{rccy3}.

For the point $s$ corresponding to $(Y,(1-a)E)$, $$K_{{\mathcal{W}_s}}+(1-a)f_*^{-1}E\equiv f_*^{-1}(K_Y+(1-a)E)\equiv 0$$ and therefore $({\mathcal{W}_s}, (1-a)f_*^{-1}E)$ is a $(2\epsilon)$-lc log Calabi--Yau pair.  
Now consider a log resolution $g: \mathcal{W}'\to \mathcal{W}$ of $(\mathcal{W},\mathcal{E})$ and denote by $\mathcal{E}'$ the strict transform of $\mathcal{E}$ and the sum of all $g$-exceptional reduced divisors on $\mathcal{W}'$. Consider the log pair $(\mathcal{W}',(1-\epsilon)\mathcal{E}')$. 
There exists an open dense set $U\subset S$ such that for the point $s\in U$ corresponding to $(Y, (1-a)E)$, 
$g_s: \mathcal{W}'_s\to \mathcal{W}_s$ is a log resolution and we can write
$$
K_{\mathcal{W}'_s}+B_s=g_s^*(K_{{\mathcal{W}_s}}+(1-a)f_*^{-1}E)\equiv 0
$$
where the coefficients of $B_s$ are $\leq 1-2\epsilon$ and its support is contained in $\mathcal{E}'_s=\mathcal{E}'|_{{\mathcal{W}'_s}}$. We have
\begin{align*}
(K_{\mathcal{W}'}+(1-\epsilon)\mathcal{E}')|_{{\mathcal{W}'_s}}
\equiv {}&K_{\mathcal{W}'_s}+(1-\epsilon)\mathcal{E}'_s\\
\equiv {}&(1-\epsilon)\mathcal{E}'_s-B_s\geq 0.
\end{align*}
Note that the support of $(1-\epsilon)\mathcal{E}'_s-B_s$ coincides with that of $\mathcal{E}'_s$ which are precisely the divisors on $\mathcal{W}'_s$ exceptional over $X$. Hence $(K_{\mathcal{W}'}+(1-\epsilon)\mathcal{E}')$ is of Kodaira dimension zero on the fiber $\mathcal{W}'_s$ and we can run a $(K_{\mathcal{W}'}+(1-\epsilon)\mathcal{E}')$-MMP with scaling of an ample divisor over $S$ to get a relative minimal model $\tilde{\mathcal{W}}$ over $S$. Such MMP terminates by \cite[Corollary 2.9, Theorem 2.12]{HX13}. Note that for the point $s\in U$ corresponding to $(Y,(1-a)E)$, $\mathcal{E}'_s$ is contracted by this MMP and hence $\tilde{\mathcal{W}}_s$ is isomorphic to $X$ in codimension one. This gives a bounded family modulo flops, over $U$.
Applying Noetherian induction on $S$, the family of all such $X$ is bounded modulo flops. 
\end{proof}


Before proving Corollary \ref{bdd index cy}, we show the boundedness of global indices in a bounded family.

\begin{lem}\label{mKY}
Let $\mathcal{D}$
 be a bounded family of projective varieties.
 Then there exists a positive integer $m$ such that if $Y\in \mathcal{D}$ is a klt Calabi--Yau variety, then $mK_Y\sim 0$.
 \end{lem}
 \begin{proof}
 Without loss of generality, we may assume that all varieties in $\mathcal{D}$ is of dimension $d$ for some positive integer $d$. 
 Note that by Global ACC  \cite[Theorem 1.5]{HMX14} (see \cite[Lemma 3.12]{rccy3}), there exists a constant $\epsilon \in (0,1)$ such that $Y$ is $\epsilon$-lc for any klt Calabi--Yau variety $Y$ in $\mathcal{D}$.

By definition, there is a quasi-projective scheme $\mathcal{Z}$ and a projective morphism $h: \mathcal{Z}\to T$, where $T$ is of finite type, such that
for every $X\in \mathcal{D}$, there is a closed point $t \in T$ and an isomorphism $f : \mathcal{Z}_t \to X$.
Replacing $T$ by disjoint union of locally closed subsets while taking log resolutions of $\mathcal{Z}$, we may assume that there   are finitely many smooth varieties $T_i$ and   projective morphisms $(\mathcal{W}_i, \mathcal{E}_i)\to \mathcal{Z}_i\to T_i$ such that $(\mathcal{W}_i, \mathcal{E}_i)$ is log smooth over $T_i$ and for every $t\in T_i$, the fiber 
 $(\mathcal{W}_{i,t}, \mathcal{E}_{i, t})$ is a log resolution of $\mathcal{Z}_{i, t}$ with $\mathcal{E}_{i, t}$ the reduced exceptional divisor, and every $X\in \mathcal{D}$ is isomorphic to a fiber of $\mathcal{Z}_i\to T_i$ for some $i$. 
 
 Note that for any $t\in T_i$ such that the fiber $\mathcal{Z}_{i, t}$ is an $\epsilon$-lc Calabi--Yau variety, and for any positive integer $m$, we have
 $$
h^0(\mathcal{W}_{i,t}, \rounddown{m(K_{\mathcal{W}_{i,t}}+(1-\epsilon)\mathcal{E}_{i, t})})=h^0( \mathcal{Z}_{i, t}, mK_{ \mathcal{Z}_{i, t}}).
$$
By \cite[Theorem 4.2]{HMX18}, the left hand side is independent of $t$ for fixed $i$ and $m$.
On the other hand, $h^0( \mathcal{Z}_{i, t}, mK_{ \mathcal{Z}_{i, t}})=1$ if and only if $mK_{ \mathcal{Z}_{i, t}} \sim 0$. Hence for each $i$, the global index of $\mathcal{Z}_{i, t}$, where $\mathcal{Z}_{i, t}$ is an $\epsilon$-lc Calabi--Yau variety, is independent of $t\in T_i$. As there are only finitely many such families, there exists a uniform positive integer $m$ such that if $Y\in \mathcal{D}$ is a klt Calabi--Yau variety, then $mK_Y\sim 0$.
\end{proof}
\begin{proof}[Proof of Corollary \ref{bdd index cy}]
Consider a  klt Calabi--Yau $3$-fold $X$. 

If $X$ has canonical singularities, then we can take a terminalization $\pi: X'\to X$ such that $X'$ has terminal singularities and $K_{X'}=\pi^*K_X$. By \cite{K=0, Morrison}, there exists a positive integer $m_1$ independent of $X'$ such that $m_1 K_{X'}\sim 0$, which implies that $m_1 K_{X}\sim 0$.

If $X$ has worse than canonical singularities, then by Theorem \ref{bdd flop cy}, $X$ is bounded modulo flops, that is, there exists a bounded family of varieties $\mathcal{D}$ such that there  is a normal projective variety $Y\in \mathcal{D}$ isomorphic to $X$ in codimension one. Moreover,  $Y$ is also a Calabi--Yau $3$-fold. Hence by Lemma \ref{mKY}, there exists a uniform positive integer $m_2$ such that $m_2K_Y\sim 0$, which implies that $m_2 K_{X}\sim 0$ as $X$ and $Y$ are isomorphic in codimension one.
\end{proof}

\section*{Acknowledgment}
The author is grateful to Valery Alexeev for  fruitful discussions. This paper (especially Section \ref{section HQ}) could never be finished without his warm encouragement and support. The author would like to thank Jungkai Alfred Chen, Jingjun Han, Yujiro Kawamata,  James M\textsuperscript{c}Kernan, Jihao Liu, Miles Reid, Vyacheslav Shokurov, and Chenyang Xu for discussions and comments during the preparation of this paper. 
Part of this paper was written during  the author's visit to Johns Hopkins University in April 2019, and the author appreciates the support and hospitality. 
 This work was supported by the National Science Foundation under Grant No.~DMS-1440140 while the author
were in residence at the Mathematical Sciences Research Institute in Berkeley, California,
during the Spring 2019 semester. The author also thanks Jihao Liu for discussions on \cite{LiuXiao}.
The author thanks the referees for carefully checking the details and many useful suggestions.


\begin{thebibliography}{99}

\bibitem{Alexeev93} V.~Alexeev, {\em Two two-dimensional terminations}, Duke Math. J. {\bf 69} (1993), no. 3, 527--545. 


\bibitem{Ale94} V.~Alexeev, \emph{Boundedness and $K^2$ for log surfaces}, Internat. J. Math. {\bf 5} (1994), no. 6, 779--810.

\bibitem{AM} V.~Alexeev, S.~Mori, \emph{Bounding singular surfaces of general
type}, Algebra, arithmetic and geometry with applications (West Lafayette, IN, 2000),
pp. 143--174, Springer, Berlin, 2004.

\bibitem{Ambro} F.~Ambro, {\em The set of toric minimal log discrepancies}, 
Cent. Eur. J. Math. {\bf 4} (2006), no. 3, 358--370. 

\bibitem{Bir16a} C.~Birkar, {\em Singularities on the base of a Fano type fibration}, J. Reine Angew. Math. {\bf 715} (2016), 125--142. 

\bibitem{BirkarBAB1} C. Birkar, Anti-pluricanonical systems on Fano varieties, Ann. of Math. {\bf 190} (2019), no. 2, 345--463.


\bibitem{Birkar18} C.~Birkar, {\em Log Calabi--Yau fibrations}, arXiv:1811.10709v2.

\bibitem{BCHM} C.~Birkar, P.~Cascini, C.D.~Hacon, J.~M\textsuperscript{c}Kernan, {\em 
Existence of minimal models for varieties of log general type},
J. Amer. Math. Soc. {\bf 23} (2010), no. 2, 405--468. 

\bibitem{BS} C.~Birkar, V.V.~Shokurov, {\em Mld's vs thresholds and flips}, 
J. Reine Angew. Math. {\bf 638} (2010), 209--234. 




\bibitem{Blache} R.~Blache, {\it The structure of l.c. surfaces of Kodaira dimension zero. I}, J. Algebraic
Geom. {\bf 4} (1995), no. 1, 137--179.


\bibitem{Borisov} A.~Borisov, {\em Minimal discrepancies of toric singularities}, Manuscripta Math. {\bf 92} (1997), no. 1, 33--45. 


\bibitem{BDPP} S.~Boucksom, J.-P.~Demailly, M.~P{\u a}un, T.~Peternell, {\em The pseudo-effective cone of a compact K\"ahler manifold and varieties of negative Kodaira dimension}, J. Algebraic Geom. {\bf 22} (2013), no. 2, 201--248. 

\bibitem{rccy3} W.~Chen, G.~Di~Cerbo, J.~Han, C.~Jiang, R.~Svaldi, {\em Birational boundedness of rationally connected Calabi--Yau $3$-folds}, arXiv:1804.09127v1.


\bibitem{DCS} G.~Di~Cerbo, R.~Svaldi, {\it Birational boundedness of low dimensional elliptic Calabi--Yau varieties with a section}, arXiv:1608.02997v2.

\bibitem{Fujiki}  A.~Fujiki, {\em On resolutions of cyclic quotient singularities}, Publ. Res. Inst. Math. Sci. {\bf 10} (1974/75), no. 1, 293--328. 

\bibitem{FG} O.~Fujino, Y.~Gongyo, {\em On canonical bundle formulas and subadjunctions}, 
Michigan Math. J. {\bf 61} (2012), no. 2, 255--264. 

\bibitem{G} Y.~Gongyo, \emph{Abundance theorem for numerically trivial log canonical divisors of semi-log canonical pairs}, J. Algebraic Geom. {\bf 22} (2013), no. 3, 549--564.



\bibitem{HM07} C.D.~Hacon, J.~M\textsuperscript{c}Kernan, {\it On Shokurov's rational connectedness conjecture},
Duke Math. J. {\bf 138} (2007), no. 1, 119--136.

\bibitem{HMX14} C.D.~Hacon, J.~M\textsuperscript{c}Kernan, C.~Xu, \emph{ACC for log canonical thresholds}, Ann. of Math. {\bf 180} (2014), no. 2, 523--571.


\bibitem{HMX18} C.D.~Hacon, J.~M\textsuperscript{c}Kernan, C.~Xu, {\em Boundedness of moduli of varieties of general type},
J. Eur. Math. Soc. (JEMS) {\bf 20} (2018), no. 4, 865--901. 


\bibitem{HX13} C.D.~Hacon, C.~Xu, \emph{Existence of log canonical closures}, Invent. Math. {\bf 192} (2013), no. 1, 161--195.


\bibitem{HLS19} J.~Han, J.~Liu, V.V.~Shokurov, {\em ACC for minimal log discrepancies of exceptional singularities}, arXiv:1903.04338v1.


\bibitem{Ishii} S. Ishii, {\em Introduction to singularities}, 
Second edition, Springer, Tokyo, 2018, x+236 pp. 



\bibitem{Kaw14} M.~Kawakita, {\em Discreteness of log discrepancies over log canonical triples on a fixed pair}, J. Algebraic Geom. {\bf 23} (2014), no. 4, 765--774. 

\bibitem{Kaw18} M.~Kawakita, {\em On equivalent conjectures for minimal log discrepancies on smooth threefolds}, arXiv:1803.02539v1. 




\bibitem{K=0} Y.~Kawamata, {\em On the plurigenera of minimal algebraic $3$-folds with $K\equiv 0$}, Math. Ann. {\bf 275} (1986), no. 4, 539--546.

\bibitem{flops} Y.~Kawamata, {\em Flops connect minimal models}, Publ. Res. Inst. Math. Sci. {\bf 44} (2008), no. 2, 419--423.

\bibitem{KMM} Y.~Kawamata, K.~Matsuda, K.~Matsuki, {\em Introduction
to the minimal model problem}, Algebraic geometry, Sendai, 1985, 283--360,
Adv. Stud. Pure Math., 10, North-Holland, Amsterdam, 1987.

\bibitem{Kollar92} J.~Koll\'ar et al, {\em Flips and abundance for algebraic threefolds}, A Summer Seminar on Algebraic Geometry (Salt Lake City, Utah, August 1991), Ast\'erisque {\bf 211} (1992), 183--192.


\bibitem{KM} J.~Koll\'{a}r, S.~Mori, \emph{Birational geometry of algebraic varieties},
Cambridge tracts in mathematics, vol. 134, Cambridge University
Press, 1998.


\bibitem{KSB} J.~Koll\'{a}r, N.I.~Shepherd-Barron, {\em Threefolds and deformations of surface singularities},
Invent. Math. {\bf 91} (1988), no. 2, 299--338. 
 
\bibitem{Liu18} J.~Liu, {\em Toward the equivalence of the ACC for $a$-log canonical thresholds and
the ACC for minimal log discrepancies}, arXiv:1809.04839v2.

\bibitem{LiuXiao} J.~Liu, L.~Xiao, {\em An optimal gap of minimal log discrepancies of threefold non-canonical singularities}, 	arXiv:1909.08759.


\bibitem{Mor18} J.~Moraga, {\em On minimal log discrepancies and Koll\'ar components}, arXiv:1810.10137v1.


\bibitem{Mori85} S.~Mori, {\em On $3$-dimensional terminal singularities}, Nagoya Math. J. {\bf 98} (1985), 43--66. 

\bibitem{Morrison} D.~Morrison, {\em A remark on Kawamata's paper ``On the plurigenera 
of minimal algebraic $3$-folds with $K\equiv 0$"}, Math. Ann. {\bf 275} (1986), no. 4, 547--553.

\bibitem{MS} D.~Morrison, G.~Stevens,{\em Terminal quotient singularities in dimensions three and four},
Proc. Amer. Math. Soc. {\bf 90} (1984), no. 1, 15--20. 


\bibitem{MN18} M.~Musta\c{t}\u{a},  Y.~Nakamura, {\em A boundedness conjecture for minimal log discrepancies on a fixed germ}, Local and global methods in algebraic geometry, 287--306, 
Contemp. Math., 712, Amer. Math. Soc., Providence, RI, 2018.


\bibitem{Nak16} Y.~Nakamura, {\em On minimal log discrepancies on varieties with fixed Gorenstein index}, Michigan Math. J. {\bf 65} (2016), no. 1, 165--187. 

\bibitem{Reid80} M.~Reid, {\em Canonical $3$-folds}, Journ\'ees de G\'eometrie Alg\'ebrique d'Angers, Juillet 1979/Algebraic Geometry, Angers, 1979, pp. 273--310, Sijthoff \& Noordhoff, Alphen aan den Rijn--Germantown, Md., 1980. 

\bibitem{YPG} M.~Reid, {\em Young person's guide to canonical
singularities}, Algebraic geometry, Bowdoin, 1985 (Brunswick, Maine, 1985), 345--414,
Proc. Sympos. Pure Math., 46, Part 1, Amer. Math. Soc., Providence, RI, 1987.


\bibitem{Sho88}
V.V.~Shokurov, {\it Problems about {F}ano varieties}, {Birational Geometry of Algebraic Varieties, Open Problems, The
	XXIIIrd International Symposium, Division of Mathematics, The Taniguchi
	Foundation}, pages 30--32, August 22 - August 27, 1988.

\bibitem{Sho91} V.V.~Shokurov, A.c.c. in codimension 2, 1991 (preprint).

\bibitem{sho-3ff}
V.V.~Shokurov, \emph{Three-dimensional log perestroikas}, Izv. Ross. Akad.
 Nauk Ser. Mat. \textbf{56} (1992), no.~1, 105--203.
 
 
 
\bibitem{MR1420223} V.V.~Shokurov, {\it {$3$}-fold log models}, Algebraic geometry, 4, J. Math. Sci. (1996), no. 3, 2667--2699.

\bibitem{letter5} V.V.~Shokurov, {\em Letters of a bi-rationalist. V. Minimal log discrepancies and termination of log flips}, 
Tr. Mat. Inst. Steklova {\bf 246} (2004), Algebr. Geom. Metody, Svyazi i Prilozh., 328--351; translation in Proc. Steklov Inst. Math. 2004, no. 3(246), 315--336.
 \bibitem{DQ1} D.-Q.~Zhang, {\em Logarithmic Enriques surfaces}, J. Math. Kyoto Univ. {\bf 31} (1991), no. 2, 419--466.

\bibitem{DQ2} D.-Q.~Zhang, {\em Logarithmic Enriques surfaces, II}, J. Math. Kyoto Univ. {\bf 33} (1993), no. 2, 357--397.
 
 \end{thebibliography}
\end{document}